\documentclass[final]{siamart171218}
\usepackage{hyperref}
\usepackage{epsfig}
\usepackage{amsfonts}
\usepackage{algorithm}
\usepackage{algorithmic}
\usepackage{graphicx}
\usepackage{subfigure}
\usepackage{array}
\usepackage{enumerate}
\usepackage{enumitem}
\usepackage{rotating}
\usepackage[T1]{fontenc}

\usepackage{pgfplots,tikz}

\newcommand\Item[1][]{%
  \ifx\relax#1\relax  \item \else \item[#1] \fi
  \abovedisplayskip=0pt\abovedisplayshortskip=0pt~\vspace*{-1.25\baselineskip}}

\newcommand{\red}{\mathsf{red}}

\newtheorem{assumption}[theorem]{Assumption}

\xdefinecolor{lavendar}{rgb}{0.4,0.2,0.6}
\xdefinecolor{lavender}{rgb}{0.92,0.9,0.97}
\xdefinecolor{paleblue}{rgb}{0.3,0.3,0.7}
\xdefinecolor{darkblue}{rgb}{0.15,0.15,0.9}
\xdefinecolor{mygreen}{rgb}{0.1,0.6,0.1}
\xdefinecolor{mydarkgreen}{rgb}{0.1,0.4,0.1}

\usepackage{todonotes}

\headers{E. Mengi}{An Interpolatory Framework for Dominant Poles}

\title{Large-Scale Estimation of Dominant Poles of a Transfer Function by an Interpolatory Framework}
\author{
Emre Mengi\thanks{Ko\c{c} University, Department of Mathematics, Rumeli Feneri Yolu 34450, Sar{\i}yer, Istanbul, Turkey, E-Mail: \texttt{emengi@ku.edu.tr}.} 
}

\begin{document}
\maketitle

\begin{abstract}
We focus on the dominant poles of the transfer function of a descriptor system. The transfer function
typically exhibits large norm at and near the imaginary parts of the dominant poles. Consequently,
the dominant poles provide information about the points on the imaginary axis where the ${\mathcal L}_\infty$ 
norm of the system is attained, and they are also sometimes useful to obtain crude reduced-order models.
For a large-scale descriptor system, we introduce a subspace framework to estimate a prescribed 
number of dominant poles. At every iteration, the large-scale system is projected into a small system,
whose dominant poles can be computed at ease. Then the projection spaces are expanded so that the 
projected system after subspace expansion interpolates the large-scale system at the computed dominant
poles. We prove an at-least-quadratic-convergence result for the framework, and provide numerical
results confirming this. On real benchmark examples, the proposed framework appears to be more
reliable than SAMDP [IEEE Trans. Power Syst. 21, 1471-1483, 2006], one of the widely used algorithms 
due to Rommes and Martins for the estimation of the dominant poles.
\end{abstract}

\begin{keywords}
dominant pole, descriptor system, large scale, projection, Hermite interpolation, model order reduction
\end{keywords}

\begin{AMS}
65F15, 93C05, 93A15, 34K17
\end{AMS}

\section{Introduction}
The dominant poles of a descriptor system provide substantial insight into the 
behavior of the transfer function of the system. Here, we consider a descriptor
system with the state-space representation
\begin{equation}\label{eq:state_space}
		E x'(t)	\;	 =	\;	A x(t)		+	B u(t),	\quad\quad
		y(t)	\;	 =	\;	C x(t)	+	D u(t),
\end{equation}
and the transfer function
\begin{equation}\label{eq:trans_func}
	H(s)	:=	C (sE - A)^{-1} B + D,	
\end{equation}
where $A, E \in {\mathbb C}^{n\times n}$, $B \in {\mathbb C}^{n\times m}$,
$C \in {\mathbb C}^{p\times n}$, $D \in {\mathbb C}^{p\times m}$ with $n \geq m$,
$\, n \geq p$. This text deals with the estimation of a few dominant poles of a 
descriptor system of the form (\ref{eq:state_space}) in the large-scale setting when 
the order of the system $n$ is large.

We assume throughout that $L(s) = A - sE$ is a regular pencil,
i.e., $\text{det} \: L(s)$ is not identically equal to zero at all $s$, and
that its finite eigenvalues are simple. If $E$ is singular, it is also assumed
that the descriptor system has index one, that is 0 as an eigenvalue of 
$L_{\mathsf{pal}}(s) = E - sA$ is semi-simple. Many of the ideas in the subsequent discussions
can possibly be generalized even if some of the finite eigenvalues of $L$ are semi-simple, 
but not necessarily simple, excluding the discussions where we analyze the
proposed framework. The index-one assumption is essential, as our interest
in the dominant poles stems partly from estimating ${\mathcal L}_\infty$ norm,
which is usually even not bounded for a system with index greater than one.

It follows from the Kronecker Canonical form \cite{Gantmacher1959} that
there exist invertible matrices $W, V \in {\mathbb C}^{n\times n}$ such that
\begin{equation}\label{eq:Kronecker}
	W^\ast	A	V	=	
	\left[
		\begin{array}{cl}
			\Lambda	&	0	\\
			0		&	I_{n - \widetilde{n}}	\\
		\end{array}
	\right]	=:	\Lambda_A
	\quad	\text{and}		\quad
	W^\ast E V	=
	\left[
		\begin{array}{cc}
			I_{\widetilde{n}}	&	0	\\
			0	&	0	\\	
		\end{array}
	\right]	=:	\Lambda_E,
\end{equation}
where $\Lambda \in {\mathbb C}^{\widetilde{n}\times \widetilde{n}}$ is the diagonal
matrix with the finite eigenvalues $\lambda_1, \dots, \lambda_{\widetilde{n}}$
of $L$ along its diagonal. Indeed, $W$ and $V$ can be expressed as
\[
	W 	=	
	\left[
		\begin{array}{cccc}
			w_1	&	\dots &	w_{\widetilde{n}}	&	W_\infty
		\end{array}
	\right]
	\quad	\text{and}		\quad
	V	=
	\left[
		\begin{array}{cccc}
			v_1	&	\dots &	v_{\widetilde{n}}	&	V_\infty
		\end{array}
	\right],
\]
where $v_j$ and $w_j$ are right and left eigenvectors corresponding to $\lambda_j$
such that $w_j^\ast E v_j = 1$ for $j = 1,\dots, \widetilde{n}$, and 
$W_\infty, V_\infty \in {\mathbb C}^{n \times {(n - \widetilde{n})}}$ have
linearly independent columns.

Using the Kronecker canonical form, transfer function (\ref{eq:trans_func}) can
be rewritten as
\begin{equation}
	\begin{split}
	H(s)		&	=	\;	C (sW^{-\ast} \Lambda_E V^{-1} - W^{-\ast} \Lambda_A V^{-1})^{-1} B + D		\\
			&	=	\;	(CV) ( s \Lambda_E	-	\Lambda_A)^{-1}	(W^\ast B)		+	D			\\
			&	=	\;
			\sum_{j=1}^{\widetilde{n}} \frac{(C v_j) (w_j^\ast B)}{s - \lambda_j}
				+	M_\infty	+	D,			\label{eq:rational}
	\end{split}
\end{equation}
where $M_\infty := -\left( C V_\infty \right) \left(  W^\ast_\infty B \right)  \in {\mathbb C}^{n\times n}$ 
is constant (i.e., independent of $s$) and due to the infinite eigenvalues of $L$.

The poles of the system are the finite eigenvalues of $L$. The dominant ones among them are those
that can cause large frequency response, i.e., the eigenvalues responsible for larger $\| H({\rm i} \omega) \|_2$
at some ${\rm i} \omega$ on the imaginary axis. The formal definition is given below.
\begin{definition}[Dominant Poles]\label{def:DP}
Let us order the finite eigenvalues $\lambda_1, \dots, \lambda_{\widetilde{n}}$ of $L$ 
as $\lambda_{i_1}, \dots, \lambda_{i_{\widetilde{n}}}$ so that
\begin{equation}\label{eq:dp_criterion}
	\frac{ \| C v_{i_1} \|_2 \| w_{i_1}^\ast B \|_2 }{| {\rm Re} \: \lambda_{i_1} |}
		\;\; \geq \;\;
	\frac{ \| C v_{i_2} \|_2 \| w_{i_2}^\ast B \|_2 }{| {\rm Re} \: \lambda_{i_2} |}
		\;\; \geq \;\;
		\dots
		\;\; \geq \;\;
	\frac{ \| C v_{i_{\widetilde{n}}} \|_2 \| w_{i_{\widetilde{n}}}^\ast B \|_2 }{| {\rm Re} \: \lambda_{i_{\widetilde{n}}} |}.
\end{equation}
The eigenvalue $\lambda_{i_j}$ is called the $j$th dominant pole of the system in (\ref{eq:state_space}).
We refer to the first dominant pole as simply the dominant pole.
\end{definition}
There are other definitions of dominant poles employed in the literature. For instance, in \cite{Antoulas2005} 
and \cite{RS2008} the dominant poles are defined based on the orderings of the finite eigenvalues 
according to $1/|{\rm Re} \: \lambda_j|$  and $\| C v_j \|_2 \| w_j^\ast B \|_2$
for $j = 1,\dots, \widetilde{n}$, respectively. It should however be noted that Definition \ref{def:DP} for the dominant 
poles that we rely on throughout this text is the one that is most widely used in the literature.
This definition also appears to be the meaningful one for instance
for model order reduction and for the estimation of $\omega \in {\mathbb R}$ where $\| H({\rm i} \omega) \|_2$
exhibits large peaks.

\subsection{Motivation}\label{sec:motivation}
A reduced order model can be obtained for (\ref{eq:state_space}) based on
the transfer function
\[
	H_{\red}(s)
		\;	=	\;
	\sum_{j=1}^{r} \frac{(C v_{i_j}) (w_{i_j}^\ast B)}{s - \lambda_{i_j}}
				+	M_\infty	+	D
\]
for a prescribed positive integer $r < \widetilde{n}$. Indeed, assuming (\ref{eq:state_space})
is asymptotically stable, the ${\mathcal H}_\infty$-norm error for this reduced order model is
\[
	\| H - H_\red \|_{{\mathcal H}_\infty}
		\;	=	\;
	\sup_{\omega \in {\mathbb R}}  \:
	\left\|
	\sum_{j=r +1}^{\widetilde{n}} \frac{(C v_{i_j}) (w_{i_j}^\ast B)}{{\rm i} \omega - \lambda_{i_j}}	\right\|_2
		\;	\leq	\;
	\sum_{j=r +1}^{\widetilde{n}} \frac{ \| C v_{i_j} \|_2  \| w_{i_j}^\ast B \|_2 }{ \left| {\rm Re} \: \lambda_{i_j} \right| }.
\]
This is known as modal model reduction. As noted in \cite[Section 9.2]{Antoulas2005}, the convergence with 
respect to the order of the reduced model may be slow. Still, if a few dominant poles can be estimated at ease,
this approach provides a crude low order approximation.

But our motivation for the estimation of the dominant poles is mainly driven from the computation of ${\mathcal L}_\infty$
norm of a descriptor system. The ${\mathcal L}_\infty$ norm for (\ref{eq:state_space}) is defined by
\begin{equation}\label{eq:Linfinity}
	\| H \|_{{\mathcal L}_\infty}	\;	:=	\;	
			\sup_{\omega \in {\mathbb R}} \: \| H({\rm i} \omega) \|_2
				\;	=	\;
			\sup_{\omega \in {\mathbb R}} \: \sigma_{\max} (C ({\rm i} \omega E - A)^{-1} B + D),
\end{equation}
where $H$ is the transfer function in (\ref{eq:trans_func}), and $\sigma_{\max}(\cdot)$
denotes the largest singular value of its matrix argument. If the system is asymptotically stable with poles
on the open left-half in the complex plane, then the ${\mathcal L}_\infty$ norm is the same as the ${\mathcal H}_\infty$ norm of 
the system, which can be expressed as an optimization problem as in (\ref{eq:Linfinity}) but with the 
supremum over the right-half of the complex plane rather than over the imaginary axis. The 
${\mathcal H}_\infty$ norm of the system is an indicator of robust stability, as indeed its reciprocal is equal to a 
structured stability radius of the system \cite{HinP86b, hinrichsen2011mathematical}. Partly due to these robust
stability considerations, if the system has design parameters, it is desirable to minimize the ${\mathcal H}_\infty$ norm 
of the system over the space of parameters; see e.g., \cite{VarP01, VizMPL16, Aliyev2020b}
and references therein. A related problem is the ${\mathcal H}_\infty$-norm model reduction
problem \cite{Antoulas2005}, which concerns finding a nearest reduced order system 
with prescribed order with respect to the ${\mathcal H}_\infty$ norm. Such minimization tasks 
involving ${\mathcal H}_\infty$ norm in the objective may require a few ${\mathcal H}_\infty$-norm calculations. 
There are various approaches that
are tailored for the estimation of the ${\mathcal L}_\infty$ norm of a large-scale descriptor system;
see for instance \cite{Guglielmi2013, Benner2014, Freitag2014, MitO15B, BennerM2018}.
However, the optimization problem in (\ref{eq:Linfinity}) is nonconvex, and
these algorithms converge to local maximizers of the singular value function in (\ref{eq:Linfinity}),
that are not necessarily optimal globally. This is for instance the case with our recent
subspace framework \cite{Aliyev2017} for large-scale ${\mathcal L}_\infty$-norm estimation,
which starts with an initial reduced order model whose transfer function interpolates the transfer
function (\ref{eq:trans_func}) of the original large-scale system at prescribed points. If these locally 
convergent algorithms are started with good initial points, in the case of \cite{Aliyev2017} good interpolation 
points, close to a global maximizer of the singular value function in (\ref{eq:Linfinity}), then 
convergence to this global maximizer occurs meaning that the ${\mathcal L}_\infty$ norm is 
computed accurately.

Good candidates for these initial points are provided by the imaginary parts of the dominant 
poles of (\ref{eq:state_space}), as global maximizers of the singular value function are typically
close to the imaginary parts of dominant poles.
This fact is illustrated in Figure \ref{fig:sval_DP},
where on the left and on the right the imaginary parts of the most dominant five and seven 
poles of the systems are depicted on the horizontal axis with crosses. They are quite close to local 
maximizers where the singular value function exhibits highest peaks. Moreover, in 
both plots in Figure \ref{fig:sval_DP}, one of these local maximizers is indeed a global maximizer.

We propose to use the imaginary parts of the dominant poles estimated by the
approach introduced here for the initialization of the algorithms for large-scale ${\mathcal L}_\infty$-norm 
computation, e.g., as the initial interpolation points for the subspace framework of \cite{Aliyev2017}.

	\begin{figure}
		\begin{tabular}{cc}
			\hskip -5ex
					\includegraphics[height=.26\textheight]{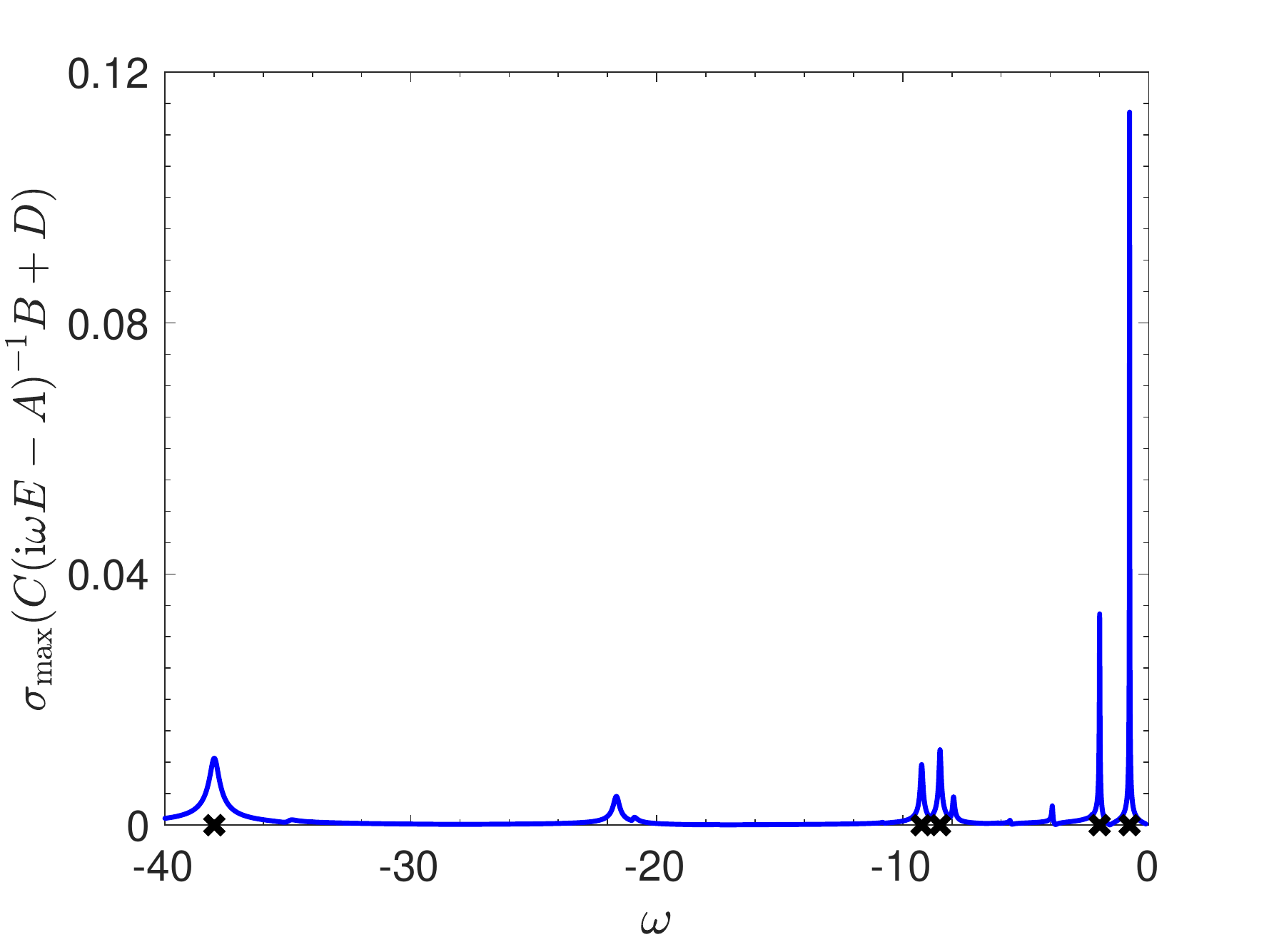} 	&	
					\hskip -6ex
							\includegraphics[height=.26\textheight]{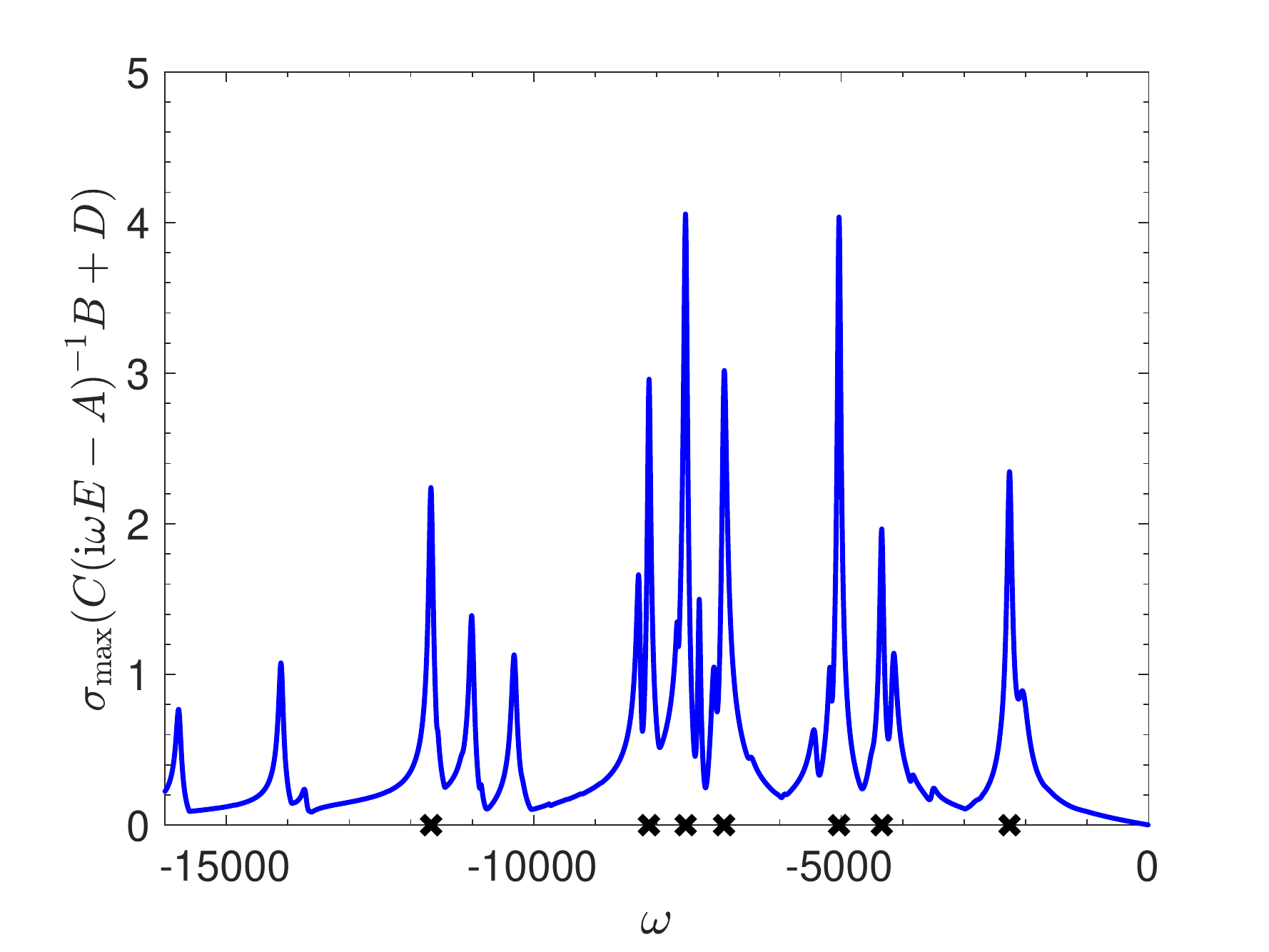}	
		\end{tabular}f
		\caption{The plots of $\sigma_{\max} (C ({\rm i} \omega E - A)^{-1} B + D)$
		as a function of $\omega$ for the \texttt{iss} system of order 270 (left) and \texttt{M10PI}$\_$\texttt{n}
		system of order 625 (right) available in the SLICOT library. The black crosses mark the imaginary parts of the
		first five (left) and first seven (right) dominant poles of the systems.}
		\label{fig:sval_DP}
	\end{figure}
	
\subsection{Literature and Our Approach}	
Several approaches have been proposed for the estimation of the dominant poles of a descriptor 
system in the literature. Most of these approaches stem from the dominant pole algorithm (DPA) \cite{MLP1996}, 
which is originally for the standard single-input-single-output LTI systems (i.e., for the descriptor systems 
as in (\ref{eq:state_space}) but with $E = I$ and $m =  p = 1$), and which is inspired 
from the Rayleigh-quotient  iteration to find an eigenvalue of $A$. DPA is meant to locate only one dominant pole.
This is later generalized to locate several dominant poles in \cite{Martins1997}. The generalized
algorithm is referred as the dominant pole spectrum eigensolver (DPSE), and can be viewed as
a simultaneous Rayleigh-quotient iteration. An extension that keeps all of the previous directions
generated by DPA and employs them for projections is described in \cite{RM2006}.
The original DPA algorithm is later adapted for multiple-input-multiple-output systems in \cite{MQ2003},
and a variant that employs all of the previous directions for projection is introduced in \cite{RM2006b}.
The extensions of these algorithms for descriptor systems with an arbitrary $E$, possibly singular, is quite straightforward;
see for instance the survey paper \cite{Rommes2008}. We also refer to \cite{RS2008} for a convergence
analysis of DPA and its variants in the single-input-single-output case. 

The subspace framework that we propose here differs from all of the existing methods for dominant
pole estimation in two major ways. First, our approach performs projections on the state-space representation, 
similar to those employed by model reduction techniques. Secondly, our approach is interpolatory. At every iteration, 
our approach computes the dominant poles of a projected system. Then it expands projection subspaces 
so as to achieve certain Hermite interpolation properties between the original system and the projected 
system at these dominant poles of the projected system. The satisfaction of the Hermite interpolation
properties is the reason for the quick convergence of the proposed framework at least at a quadratic rate
under mild assumptions, which we prove in theory and illustrate in practice on real examples.

The proposed framework is partly inspired from our previous work for ${\mathcal L}_\infty$-norm
computation \cite{Aliyev2017}. However, devising a rapidly converging subspace framework
for dominant pole estimation and establishing its quick convergence come with additional
challenges. By making an analogy with quasi-Newton methods,
in \cite{Aliyev2017} we maximize over the imaginary axis the largest singular value of the 
reduced transfer function, viewed as a model function, rather than the full transfer function. In the context here, 
there is no clear objective to be maximized. Instead, the quantities that 
we would like to compute, i.e., dominant poles, are hidden inside the transfer function. Thus, to explain 
the quick convergence and provide a formal rate-of-convergence argument, we devise a function 
(in particular $f$ in (\ref{eq:defn_f})) 
over the complex plane whose 
minimizers are poles of $H(s)$, and view its reduced counterparts as model functions. 
The devised function and their reduced counterparts are tailored as analytic functions near the poles
under mild assumptions; the rate-of-convergence analysis we present would not be applicable without
such smoothness features. A second critical 
issue is that all the interpolation points and analysis in \cite{Aliyev2017} are restricted to the imaginary axis, 
while, in the current work, all of the developments have to be over the whole complex plane. Tailoring a 
subspace framework over the complex plane requires care. For instance, it seems essential to interpolate 
at least the first two derivatives of the transfer functions in order to attain superlinear convergence, 
unlike the framework in \cite{Aliyev2017} operating on the imaginary axis for which interpolating
first derivatives suffices for superlinear convergence. The proposed framework has also similarities
with those in \cite{Aliyev2020b} and \cite{AMV2020} for the minimization of the ${\mathcal H}_\infty$ norm
and for nonlinear eigenvalue problems. But again there are remarkable differences in the ways
interpolation is performed, and in the rate-of-convergence analyses explaining quick convergence; 
to say the least, in those works the analytic functions on which the analyses operate on are
readily available, and quite different than the one we set up here.

\subsection{Contributions and Outline}
We introduce the first interpolatory subspace framework for the estimation of
the dominant poles of a large-scale descriptor system. The framework appears
to be more reliable than the existing methods. On benchmark examples
the framework proposed here typically returns dominant poles with larger dominance
metrics compared to those dominant poles returned by the method in \cite{RM2006b},
commonly employed today for dominant pole estimation. We prove rigorously
that the rate of convergence of the framework is at least quadratic  with
respect to the number of subspace iterations, which is also confirmed
in practice on benchmark examples. The proposed framework is implemented
rigorously, and this implementation is made publicly available. The interpolation
result presented here for the transfer functions of descriptor systems (i.e., Lemma \ref{thm:main_interpolation}) 
generalizes the interpolation results in our previous works \cite[Lemma 3.1]{Aliyev2017}, \cite[Lemma 2.1]{AMV2020}.

Our presentation is organized as follows. The next section concerns an application of the dominant poles for 
estimating the stability radius of a linear dissipative Hamiltonian system, 
a problem closely connected to ${\mathcal L}_\infty$ norm.
We describe the interpolatory subspace framework to compute a prescribed number of
dominant poles in Section \ref{sec:framework}. In Section \ref{sec:Rconv},
the rate of convergence of the proposed subspace framework to compute the most dominant pole 
is analyzed. Section \ref{sec:prac_details} is devoted to the practical details that have to be
taken into consideration in an actual implementation of the framework such as initialization and
termination criterion. The proposed framework is tested numerically in Section \ref{sec:NR}
on real benchmark examples used in the literature for dominant pole estimation and model
order reduction. In this numerical experiments section, comparisons of the framework with the subspace 
accelerated multiple-input-multiple-output dominant pole algorithm \cite{RM2006b} are reported.

\section{Stability Radius of a Linear Dissipative Hamiltonian System}\label{sec:stab_rad}
Computation of the ${\mathcal L}_\infty$-norm (\ref{eq:Linfinity}) for a large-scale system
efficiently and with high accuracy is still not fully addressed today. 

The level-set method due to Boyd and Balakrishnan \cite{Boyd1990}, Bruinsma and Steinbuch \cite{Bruinsma1990}
is extremely reliable and accurate. It is still the method to be used for a small-scale system.
Unfortunately, it is not meant for large-scale systems, as, for a system of order $n$, it
requires the calculation of all imaginary eigenvalues of $2n\times 2n$ matrices. The only
other option is to use the locally convergent algorithms such as \cite{Guglielmi2013, Benner2014, Freitag2014, MitO15B, Aliyev2017, BennerM2018}.
These algorithms are better suited to cope with the large-scale setting,
but, especially when the singular value function in (\ref{eq:Linfinity}) has many local
maximizers, there is a decent chance that they will converge to a local
maximizer that is not optimal globally. A particular problem that is connected to 
the computation of an ${\mathcal L}_\infty$ norm is the stability radius of a linear 
dissipative Hamiltonian (DH) system.

A linear DH system is a linear autonomous control system of the form
\[
	x'(t)	\;	=	\;	(J-R) Q x(t),
\]
where $J, R, Q \in {\mathbb C}^{\widetilde{n} \times \widetilde{n}}$ are constant matrices such that $J^\ast = -J$,
$R^\ast = R$, $Q^\ast = Q$, and $R$, $Q$ are positive semidefinite, positive definite, respectively.
Various applications in science and engineering give rise to linear DH systems \cite{JacZ12, SchJ14, GraMQSW16}.
A linear DH system is always Lyapunov stable, that is all of the eigenvalues of $(J-R)Q$ are contained
on the closed left-half plane and its eigenvalues on the imaginary axis (if there is any) are semi-simple.
However, unstructured perturbations of $J, R$ and/or $Q$ may result in unstable systems with
eigenvalues whose real parts are positive. In the presence of uncertainties on the matrix $R$, 
for given restriction matrices $B \in {\mathbb C}^{\widetilde{n}\times \widetilde{m}}, \, C \in {\mathbb C}^{\widetilde{p}\times \widetilde{n}}$ with $\widetilde{m} \leq \widetilde{n}$, 
$\, \widetilde{p} \leq \widetilde{n}$ on the perturbations of $R$ due to uncertainties, the stability radius
\[
	r(R; B,C) := \inf \{\|\Delta\|_2 \; | \; \Delta\in \mathbb{C}^{\widetilde{m} \times \widetilde{p}},  \:
	  \mathrm{\Lambda} \left( (J - \left(R+B\Delta C\right)) Q \right) \cap {\rm i}\mathbb{R} \neq \emptyset\}
\]
is proposed in \cite{MehMS16a}, where $\Lambda(\cdot)$ denotes the spectrum of its matrix argument, and ${\rm i} {\mathbb R}$
the set of purely imaginary numbers. It is also proven in \cite{MehMS16a} that $r(R; B,C)$ can be characterized as
\begin{equation}\label{eq:DH_char}
	r(R; B,C)
		\;	=	\;
		\frac{1}{\sup_{\omega\in {\mathbb R}} h(\omega)},	\quad\quad
		h(\omega) 	\: 	:= 	\:	 \sigma_{\max} (CQ ({\rm i} \omega I - (J-R)Q)^{-1} B)
\end{equation}
provided that $r(R; B,C)$ is finite. Hence,
$r(R; B, C)$ is the reciprocal of a special ${\mathcal L}_\infty$ norm as in (\ref{eq:Linfinity}) with
$CQ$, $(J-R)Q$ taking place of $C$, $A$ and $E = I$, $D = 0$.

The singular value function in (\ref{eq:DH_char}) has usually many local maximizers especially when $R$
has low rank. 
As argued in Section \ref{sec:motivation}, a remedy to convergence to local maximizers that are not optimal globally
is to use the dominant poles of the associated system
\begin{equation}\label{eq:DHsys}
		 x'(t)	\;	 =	\;	(J-R)Q x(t)		+	B u(t),	\quad\quad
		y(t)	\;	 =	\;	CQ x(t)
\end{equation}
for initialization. To illustrate this point, we compute $r(R; B, C)$ for 3000 random linear DH systems with $\widetilde{n} =500$, 
$\widetilde{m} = \widetilde{p} = 2$\footnote{Such random linear DH systems can be generated using the Matlab routine \texttt{randomDH} made available
on the web at \url{https://zenodo.org/record/5103430}, specifically with the command \texttt{randomDH(500,2,2)}.}
 using the subspace framework \cite{Aliyev2017} and the framework \cite{Aliyev2020}, the variant that respects 
the DH structure, both of which are locally convergent. For all of these DH systems, $R$ is constrained to have rank in $[10,50]$,
and the singular value function $h(\omega)$ in (\ref{eq:DH_char}) has typically at least thirty local maximizers.
As for the initial interpolation points, we consider the following two possibilities: \\[.4em]
\begin{tabular}{l}
	\textbf{(equally-spaced points)} 10 points among 40 equally-spaced points in $[-1200, 0]$ \\[.2em]
	(which contains a global maximizer of $h(\omega)$) yielding the largest value of $h(\omega)$; \\[.3em]
	\textbf{(dominant poles)} 10 points among the imaginary parts of the 40 most dominant \\[.2em]
	poles of (\ref{eq:DHsys}) that yield the largest value of $h(\omega)$.
\end{tabular}	

\medskip

\noindent
Four approaches are compared in Table \ref{tab:DH_stabradii}. The BB, BS algorithm in the last line of the table refers to the level-set 
method due to Boyd and Balakrishnan \cite{Boyd1990}, Bruinsma and Steinbuch \cite{Bruinsma1990}.
Even the subspace frameworks \cite{Aliyev2017, Aliyev2020} benefit from this level-set method to solve the small projected problems. 
We remark that $n = 500$ is relatively small so that the BB, BS algorithm is applicable, in particular
we can verify the correctness of the results computed by the frameworks by comparing them with the results returned by the BB, BS algorithm. 
It is apparent from the second column in Table \ref{tab:DH_stabradii} that \cite{Aliyev2020} initialized with dominant poles is 
considerably more accurate than \cite{Aliyev2017, Aliyev2020} using equally-spaced points. The runtimes
reported in the last column include the time for initialization, in particular the time for the computation of the dominant poles
in the third line. Observe that \cite{Aliyev2020} with dominant poles is substantially faster than direct applications 
of the level-set method \cite{Boyd1990, Bruinsma1990}, and nearly as accurate as the level-set method.

\begin{table}
\caption{ Comparison of the subspaces frameworks \cite{Aliyev2020} with dominant poles and \cite{Aliyev2017, Aliyev2020}
with equally-spaced points on 3000 random linear DH systems in Section \ref{sec:stab_rad}. Accuracy refers to the 
percentage of the results that differ from the results by the BB, BS algorithm by an amount less than $10^{-8}$.} 
 \label{tab:DH_stabradii}
 \begin{center}
    	\begin{tabular}{l|c|cc|cc}
								 		&						&	 \multicolumn{2}{c}{$\#$ iterations}				&	 \multicolumn{2}{c}{time in s}	\\
		\phantom{aaaaaaaaa} method			&	accuracy				&		mean & median							&		mean & median		\\
						\hline
\cite{Aliyev2017} with equally-spaced points 		&	 $\%$ 84.93			&		17.90 & 13		&		\phantom{1}1.86  &  	\phantom{1}0.69		\\[.3em]
\cite{Aliyev2020} with equally-spaced points 		&	 $\%$ 91.93			&		12.90 & 11		&		\phantom{1}2.54  &  	\phantom{1}1.25		\\[.3em]
\cite{Aliyev2020} with dominant poles			&	 $\%$ 99.43			&		10.52 &  9			&		\phantom{1}1.91  &	\phantom{1}1.00		\\[.3em]
BB, BS Algorithm \cite{Boyd1990, Bruinsma1990}	&	 $\%$ 100				&		---	 &  ---		&		\phantom{1}6.61  &	\phantom{1}6.39		\\
    	\end{tabular}
\end{center}
\end{table}

\section{The Proposed Subspace Framework}\label{sec:framework}
Some of the most widely-used approaches for model order reduction of descriptor systems perform
Petrov-Galerkin projections. In the Petrov-Galerkin framework, given two subspaces ${\mathcal V}$, ${\mathcal W}$
and matrices $V$, $W$ whose columns form orthonormal bases for these subspaces, system (\ref{eq:state_space})
is approximated by
\[
		W^\ast E V x_\red'(t)	\;	 =	\;	W^\ast A V x_\red(t)		+	W^\ast B u(t),	\quad\quad
		y(t)	\;	 =	\;	C V x_\red(t)	+	D u(t).
\]
Note that this reduced order system is obtained from (\ref{eq:state_space}) by restricting the state space to
${\mathcal V}$ (i.e., by replacing $x(t)$ with $V x_\red(t)$, which is merely an approximation) and imposing the 
orthogonality of the resulting residual of the differential part to ${\mathcal W}$. Moreover, it is important that
the dimensions of ${\mathcal V}$ and ${\mathcal W}$ are the same, so that $W^\ast E V$ and $W^\ast A V$ are
square matrices, and poles of the reduced system are the eigenvalues of a square pencil, just like the
original system.

Interpolation is a plausible strategy for the construction of the subspaces; ${\mathcal V}$, ${\mathcal W}$ can be
formed so that the transfer function of the reduced system  
\begin{equation}\label{eq:red_transfer_function}
	H^{{\mathcal W}, {\mathcal V}}_\red(s)	:=	CV (sW^\ast E V - W^\ast A V)^{-1} W^\ast B + D
\end{equation}
Hermite interpolates the transfer function $H(s)$ as in (\ref{eq:trans_func}) of the original system at prescribed points.
The following result indicates how Hermite interpolation properties between the full 
and reduced transfer functions can be attained at prescribed points. Note that $I_m$ and $I_p$ stand for $m\times m$
and $p\times p$ identity matrices, respectively.   
\begin{lemma}\label{thm:main_interpolation}
Let $\mu \in {\mathbb C}$ be a point that does not belong to the spectrum of $L(s) = A - s E$.
Furthermore, let ${\mathcal W}_{\rm up} = {\mathcal W} \oplus {\mathcal W}_\mu$ and 
${\mathcal V}_{\rm up} = {\mathcal V} \oplus {\mathcal V}_\mu$ for two subspaces ${\mathcal V}, {\mathcal W}$
of equal dimension, and ${\mathcal V}_\mu$, ${\mathcal W}_\mu$ defined as 
		\begin{equation*}
		   \begin{split}
			{\mathcal V}_\mu  & \;  :=  \; \bigoplus_{j=0}^{q} \: \mathrm{Ran} \left[ \left( \left\{ (A - \mu E)^{-1} E \right\}^j (A - \mu E)^{-1} B \right) P_R(\mu) \right] \: , \\
			 {\mathcal W}_\mu  &  \; :=  \; \bigoplus_{j=0}^{q} \: \mathrm{Ran} \left[ \left( C (A - \mu E)^{-1} \left\{ E (A - \mu E)^{-1} \right\}^j \right)^\ast P_L(\mu) \right]
		   \end{split}
		\end{equation*}
for some integer $q \geq 0$, and
\begin{equation}\label{eq:defn_PR_PL}
	P_{R}(\mu)
		\;	=	\;
	\left\{
	\begin{array}{cc}
		I_m		\quad\quad	&		\text{if} \;\; m \leq	p	\\[.2em]
		H(\mu)^\ast	\quad \quad	&	\text{if} \;\;  m	>	p
	\end{array}
	\right.	\;\;	,		\quad\quad
	P_{L}(\mu)
		\;	=	\;
	\left\{
	\begin{array}{cc}
		I_p		\quad\quad	&		\text{if} \;\;  p \leq	m	\\[.2em]
		H(\mu)	\quad \quad	&	\text{if} \;\;  p	>	m
	\end{array}
	\right.	\;	.
\end{equation}
If $\mu$ is not a pole of $H^{{\mathcal W}_{\rm up}, {\mathcal V}_{\rm up}}_\red$, then we have \\[-.5em]
	\begin{enumerate}
		\item \hskip -.7ex
		$H(\mu) = H^{{\mathcal W}_{\rm up}, {\mathcal V}_{\rm up}}_\red(\mu)$,			
		\vskip .5ex
		\item \hskip -.7ex 
		$H^{(j)}(\mu) = \big[ H^{{\mathcal W}_{\rm up}, {\mathcal V}_{\rm up}}_\red \big]^{(j)} (\mu) \:$	
		for $j = 1,\dots, q$,  $\;$ and
		\vskip .5ex
		\item \hskip -.7ex 
		$P_L(\mu)^\ast H^{(j)}(\mu) P_R(\mu) = P_L(\mu)^\ast \big[ H^{{\mathcal W}_{\rm up}, {\mathcal V}_{\rm up}}_\red \big]^{(j)} (\mu) P_R(\mu)$	
		for $j = q+ 1,\dots, 2q + 1$,
	\end{enumerate}
\vskip .5ex
where $H^{(j)}$, $\big[ H^{{\mathcal W}_{\rm up}, {\mathcal V}_{\rm up}}_\red \big]^{(j)}$ denote the $j$th derivatives of 
$H$, $H^{{\mathcal W}_{\rm up}, {\mathcal V}_{\rm up}}_\red$, respectively.	
\end{lemma}
\begin{proof}
For every $j \in \{ 0 , \dots , q \}$, we have
\begin{equation*}
	\begin{split}
		&
			 \frac{d^j}{ds^j}  \left\{ (A - s E)^{-1} B \right\}	\bigg|_{s = \mu} 
		 	\:		=	\:
				\left\{ (A - \mu E)^{-1} E \right\}^j (A - \mu E)^{-1} B 	\,	,		 \quad	\;\;		\text{and}		\\[.2em]
		&
			 \frac{d^j}{ds^j}  \left\{ C (A - s E)^{-1}   \right\} \bigg|_{s = \mu} 
		 	\:		=	\:
				C (A - \mu E)^{-1} \left\{ E (A - \mu E)^{-1} \right\}^j	 \, .
	\end{split}
\end{equation*}
Hence, letting $r := \min \{ p , q \}$, for every $v, w \in {\mathbb C}^r$ it follows from the definitions of ${\mathcal V}_m$ and ${\mathcal W}_m$ that
\begin{equation*}
\begin{split}
	&	 \left( \frac{d^j}{ds^j}  \left\{ (A - s E)^{-1} B  \right\} \bigg|_{s = \mu} \right)  P_{R}(\mu) v	\: \in \: {\mathcal V}_\mu	\,	,	\quad	\;\;	\text{and}	\\[.2em]
	&	 \left( \frac{d^j}{ds^j}  \left\{ C (A - s E)^{-1}   \right\} \bigg|_{s = \mu} \right)^\ast  P_{L}(\mu) w	\: \in \: {\mathcal W}_\mu
\end{split}
\end{equation*}
for $j = 0, 1, \dots , q$. Now \cite[Theorem 1]{Gugercin2009} implies that for every $v, w \in {\mathbb C}^r$ we have
\begin{eqnarray}
	&	H^{(j)}(\mu)  \, 	(P_R(\mu) v)
				\;\;	 =	\;\;
		 \big[ H^{{\mathcal W}_{\rm up}, {\mathcal V}_{\rm up}}_\red \big]^{(j)} (\mu) \, 	(P_R(\mu) v)			\label{eq:interpolate1}	\\
	&	(P_L(\mu) w)^\ast	\,  H^{(j)}(\mu)  
				\;\;	 =	\;\;
		(P_L(\mu) w)^\ast  \,  \big[ H^{{\mathcal W}_{\rm up}, {\mathcal V}_{\rm up}}_\red \big]^{(j)} (\mu)	\label{eq:interpolate2}
\end{eqnarray}
for $j	=	0, 1, \dots , q$, and
\begin{equation}\label{eq:interpolate3}
		(P_L(\mu) w)^\ast	\,  H^{(j)}(\mu)  \, 	(P_R(\mu) v)
				\;\;	=	\;\;
		(P_L(\mu) w)^\ast  \,  \big[ H^{{\mathcal W}_{\rm up}, {\mathcal V}_{\rm up}}_\red \big]^{(j)} (\mu) \, 	(P_R(\mu) v)
\end{equation}
for $j	 = 0, 1, \dots , 2q+1$.

We consider the cases $m = p$, $m > p$ and $m < p$ separately below. Throughout the rest of this proof, 
recall that $r := \min \{ p , q \}$. 

\smallskip

\noindent
\textit{Case 1 ($m = p$)} : In this case, $P_R(\mu)$ and $P_L(\mu)$ are the identity matrices. 
Let $j \in \{0, \dots, 2q+1 \}$. For every $v, w \in {\mathbb C}^r$,
we deduce from (\ref{eq:interpolate3}) that
\[
			w^\ast	\,  H^{(j)}(\mu)  \, 	v
				\;\;	=	\;\;
			w^\ast  \,  \big[ H^{{\mathcal W}_{\rm up}, {\mathcal V}_{\rm up}}_\red \big]^{(j)} (\mu) \, 	v
\]
holds for every $v, w \in {\mathbb C}^r$. This in turn implies 
$H^{(j)}(\mu)	=	\big[ H^{{\mathcal W}_{\rm up}, {\mathcal V}_{\rm up}}_\red \big]^{(j)} (\mu)$.

\medskip

\noindent
\textit{Case 2 ($m > p$)} : In this case, $P_R(\mu) = H(\mu)^\ast$ and $P_L(\mu) = I_p$.
For $j  = 0, 1, \dots, q$, denoting with $e_\ell$ the $\ell$th column of $I_p$,
it follows from (\ref{eq:interpolate2}) that
\[
		e_\ell^\ast	\,  H^{(j)}(\mu)  
				\;\;	 =	\;\;
		e_\ell^\ast  \,  \big[ H^{{\mathcal W}_{\rm up}, {\mathcal V}_{\rm up}}_\red \big]^{(j)} (\mu)
\]
for $\ell = 0, 1, \dots, p \,$ so that $\, H^{(j)}(\mu)  =  \big[ H^{{\mathcal W}_{\rm up}, {\mathcal V}_{\rm up}}_\red \big]^{(j)} (\mu)$.

Now let $j \in \{ q+1, \dots , 2q+1 \}$. For every $v, w \in {\mathbb C}^r$, by (\ref{eq:interpolate3}), 
we have
\begin{equation}\label{eq:higher_ders}
		w^\ast P_L(\mu)^\ast H^{(j)}(\mu) P_R(\mu) v  \;\; = \;\; w^\ast P_L(\mu)^\ast \big[ H^{{\mathcal W}_{\rm up}, {\mathcal V}_{\rm up}}_\red \big]^{(j)} (\mu) P_R(\mu) v.
\end{equation}
This shows that $P_L(\mu)^\ast H^{(j)}(\mu) P_R(\mu)	=	P_L(\mu)^\ast \big[ H^{{\mathcal W}_{\rm up}, {\mathcal V}_{\rm up}}_\red \big]^{(j)} (\mu) P_R(\mu)$ as desired.

\medskip

\noindent
\textit{Case 3 ($m < p$)} : As $P_R(\mu) = I_m$ and $P_L(\mu) = H(\mu)$,
for $j  = 0, 1, \dots, q$ the identity in (\ref{eq:interpolate1}) implies
\[
		 H^{(j)}(\mu)  \, e_\ell
				\;\;	 =	\;\;
		 \big[ H^{{\mathcal W}_{\rm up}, {\mathcal V}_{\rm up}}_\red \big]^{(j)} (\mu)  \, e_\ell
\]
for $\ell = 0, 1, \dots, m$, where  $e_\ell$ is the $\ell$th column of $I_m$. Consequently, we deduce
$H^{(j)}(\mu)  =  \big[ H^{{\mathcal W}_{\rm up}, {\mathcal V}_{\rm up}}_\red \big]^{(j)} (\mu)$
for $j  = 0, 1, \dots, q$.

For $j \in \{ q+1, \dots , 2q+1 \}$, once again the equality in (\ref{eq:higher_ders}) holds
for every $v, w \in {\mathbb C}^r$ by (\ref{eq:interpolate3}),
implying $P_L(\mu)^\ast H^{(j)}(\mu) P_R(\mu)	=	P_L(\mu)^\ast \big[ H^{{\mathcal W}_{\rm up}, {\mathcal V}_{\rm up}}_\red \big]^{(j)} (\mu) P_R(\mu)$.
\end{proof}
The result above holds even when $P_R(\mu)$ and $P_L(\mu)$ are replaced by the identity 
matrices of proper sizes, but then, unless $m = p$, the matrices $( \, \{ (A - \mu E)^{-1} E \}^j (A - \mu E)^{-1} B \, )$ and 
$( \, C (A - \mu E)^{-1} \{ E (A - \mu E)^{-1} \}^j \, )^\ast$ do not have equal number of columns. Hence, the role of 
$P_R(\mu)$ and $P_L(\mu)$ is to make sure ${\mathcal V}_m$  and ${\mathcal W}_m$ are defined in terms of 
matrices with equal number of columns and of usually full column rank so that the dimensions of
${\mathcal V}_m$  and ${\mathcal W}_m$ are usually the same.



Our proposed framework, built on these projection and interpolation ideas, operates as follows. It computes the dominant poles of a 
reduced system at every iteration. As the reduced systems are of
low order, in practice this is achieved by first computing all of the finite poles of the reduced system, for instance by using
the QZ algorithm, and sorting them from the largest to the smallest based on the dominance metric in (\ref{eq:dp_criterion}). 
Then, the framework expands the subspaces so that the transfer function of the reduced system after expansion
Hermite interpolates the transfer function of the full system at the computed dominant poles of the reduced system before expansion. 
For this interpolatory subspace expansion task, we put Lemma \ref{thm:main_interpolation} in use.
The details of the proposed framework are given in Algorithm \ref{alg:SM}, where the dominant poles of the reduced 
system are computed in line \ref{inter_point}, and the subspaces are expanded in lines \ref{sub_expand_st0}-\ref{sub_expand_end}
so as to satisfy Hermite interpolation at the points selected among these computed dominant poles 
of the reduced system based on whether they are yet to converge to actual poles. 
We postpone the discussions of how the initial subspaces are constructed in line \ref{init_subspaces}, the termination condition 
to assert convergence in line \ref{return_spec}, and the criterion to decide whether the estimates $\lambda_\ell^{(j)}, v_\ell^{(j)}$ 
have converged in line \ref{check_convergence} to Section \ref{sec:prac_details}. A subtle issue is that we require the
interpolation parameter $q$ in Algorithm \ref{alg:SM} to satisfy $q \geq 1$ if $m = p$ and $q \geq 2$ otherwise.
These choices of $q$ (by Lemma \ref{thm:main_interpolation}) ensure that the first two derivatives of the full transfer function are 
interpolated by those of the reduced transfer function at the interpolation points, which is exploited by the quadratic convergence 
proof in the next section for Algorithm \ref{alg:SM}.

\begin{algorithm}[tb]
 \begin{algorithmic}[1]
 
\REQUIRE{system matrices $A, E \in {\mathbb C}^{n\times n}$, $B \in {\mathbb C}^{n\times m}$, $C \in {\mathbb C}^{p\times n}$, 
	$D \in {\mathbb C}^{p\times m}$ for the descriptor system in (\ref{eq:state_space}), interpolation parameter $q \in {\mathbb Z}$
	such that $q \geq 1$ if $m = p$ and $q \geq 2$ if $m \neq p$, and number of dominant poles sought $\kappa$}
\ENSURE{estimate $\zeta_j \in {\mathbb C}$ for the $j$th dominant pole of (\ref{eq:state_space}) and the corresponding eigenvector 
estimate $z_j$ of $L(s) = A - s E$ for $j = 1, 2, \cdots, \kappa$}

\vskip 1.1ex

\STATE \textcolor{mygreen}{\textbf{$\%$ form the initial subspaces}} \\
Set matrices$\,$\footnotesize$V_0$, $W_0$\normalsize$\,$whose columns form orthonormal bases for initial subspaces.   \label{init_subspaces}  

\vskip 1.1ex

\FOR{$\ell = 1,\,2,\,\dots$}

 	\vskip 1.1ex
	
	\STATE \textcolor{mygreen}{\textbf{$\%$ update the estimates for the dominant poles}} \\
	 $\lambda^{(j)}_{\ell}, v^{(j)}_\ell \gets j\text{th dominant pole of the system with transfer func. }
	 H^{{\mathcal W}_{\ell -1}, {\mathcal V}_{\ell -1}}_\red(s)$ 
	 and corresponding eigenvector of $L^{W_{\ell-1}, V_{\ell-1}}(s) = W^\ast_{\ell-1} A V_{\ell-1} - s W^\ast_{\ell-1} E V_{\ell-1}$ $\text{for } j = 1,\dots, \kappa$, where
	 ${\mathcal W}_{\ell -1}, {\mathcal V}_{\ell -1}$ are the subspaces spanned by the columns of $W_{\ell - 1}$, $V_{\ell - 1}$.		\label{inter_point}
	 
	 \vskip 1.1ex
	 
	 \STATE  \textcolor{mygreen}{\textbf{$\%$ terminate in the case of convergence}} \\
	  \textbf{Return} $\zeta_j \gets \lambda^{(j)}_{\ell}$,
		$\: z_j \gets V_{\ell-1} v^{(j)}_{\ell}$ for $j = 1,\dots, \kappa$ if convergence occurred.  \label{return_spec}
		
	\vskip 1ex		
		
	 \STATE \textcolor{mygreen}{\textbf{$\%$ lines 5-15: $\,$expand subspaces to interpolate at $\lambda^{(j)}_{\ell}$, $\: j = 1,\dots,\kappa$}} \\
	 $V_\ell \gets V_{\ell-1} 
			\quad \text{and}\quad W_\ell \gets W_{\ell-1}$. \label{sub_expand_st0}
	\vskip .3ex
	 \FOR{$j = 1, \dots, \kappa$} \label{sub_expand_st}
	 \vskip .3ex
	 \IF{$\lambda^{(j)}_{\ell}, v^{(j)}_\ell$ did not converge up to the prescribed tolerance}  \label{check_convergence}
	 \vskip .3ex
	  \STATE $\widehat{V} \gets (A - \lambda_\ell^{(j)} E)^{-1} B P_R(\lambda^{(j)}_\ell)$,
		  		$\; \widetilde{V} \gets  \widehat{V} $, $\; \widehat{W} \gets (A - \lambda_\ell^{(j)} E)^{-\ast} C^\ast P_L(\lambda^{(j)}_\ell)$
				and $\widetilde{W} \gets \widehat{W} $, where $\; P_R(\lambda^{(j)}_\ell), P_L(\lambda^{(j)}_\ell) \;$
				are as in (\ref{eq:defn_PR_PL}).   \label{exp_start}
	\vskip .3ex	
			\FOR{$r = 1, \dots, q$}
	\vskip .3ex
		  \STATE $\widehat{V} \gets (A - \lambda^{(j)}_\ell E)^{-1} E \widehat{V} \;$ and  
		  		$\; \widetilde{V} \gets \begin{bmatrix} \widetilde{V} & \widehat{V} \end{bmatrix}$.	\label{exp_start2}
	\vskip .3ex
		  \STATE	$\widehat{W} \gets (A - \lambda^{(j)}_\ell E)^{-\ast} E^\ast \widehat{W} \;$		
						and
				$\; \widetilde{W} \gets \begin{bmatrix} \widetilde{W} & \widehat{W} \end{bmatrix}$.	\label{exp_start3} 
		        \ENDFOR 			       		        
	\vskip .3ex			        
		\STATE
		$V_\ell \gets \operatorname{orth}\left(\begin{bmatrix} V_{\ell} & \widetilde{V} \end{bmatrix}\right)
			\quad \text{and}\quad W_\ell \gets \operatorname{orth}\left(\begin{bmatrix} W_{\ell} & \widetilde{W} \end{bmatrix}\right).$  \label{orthogonalize}
	\vskip .3ex	
	\ENDIF	
	\vskip .3ex				
	\ENDFOR    \label{sub_expand_end}
	
	\vskip 1ex
	
\ENDFOR
 \end{algorithmic}
\caption{Subspace framework to compute dominant poles of (\ref{eq:state_space})}
\label{alg:SM}
\end{algorithm}

The framework resembles the one that we have introduced for ${\mathcal L}_\infty$-norm computation in \cite{Aliyev2017}. 
Only here the dominant poles of the reduced systems are used as the interpolation points, whereas in \cite{Aliyev2017} the 
interpolation points are chosen on the imaginary axis as the points where the ${\mathcal L}_\infty$ norm of the reduced system 
is attained. A second remark is that rather
than interpolating at all dominant poles of the reduced system, we could interpolate 
at only one of the dominant poles at every iteration, e.g., the most
dominant one among the dominant poles of the reduced system that are yet to converge. We have explored such alternatives
in \cite{AMV2020} in the context of computing a few nonlinear eigenvalues closest to a prescribed target. Our numerical experience
is that interpolating at all eigenvalues of the reduced problem is usually more reliable, 
even though in some cases one-per-subspace-iteration interpolation strategy may be more efficient.  

Another possible variation of the subspace framework outlined in Algorithm \ref{alg:SM} is setting the left-hand projection
subspace equal to the right-hand subspace, which is commonly referred as a \textit{one-sided} subspace framework
in model reduction. In contrast, a subspace framework such as Algorithm \ref{alg:SM} that employs different
subspaces from left and right is referred as a \textit{two-sided} subspace framework. In a one-sided variation of
Algorithm \ref{alg:SM}, the right-hand subspace ${\mathcal V}_\ell$ at iteration $\ell$ is formed as in Algorithm \ref{alg:SM},
however ${\mathcal W}_\ell = {\mathcal V}_\ell$. An analogue of the Hermite interpolation result (i.e., Lemma \ref{thm:main_interpolation})
still holds; full and reduced transfer functions are equal at the interpolation point, but their 
derivatives match only up to the $q$th derivative. In the one-sided setting, provided $q \geq 2$,
the rate-of-convergence analysis in the next section still applies leading to the quick convergence 
result in Theorem \ref{thm:quick_convergence}. The advantages of a one-sided variation are that an orthonormal
basis for only one subspace needs to be kept, and, with $q = 2$, fewer number of back and forward
substitutions per iteration may be required while retaining quick convergence. However, in our experience,
it is numerically less stable than its two-sided counterpart. 

\section{Rate of Convergence of the Subspace Framework}\label{sec:Rconv}
We consider Algorithm \ref{alg:SM} when only the dominant pole is sought (i.e., with $\kappa = 1$). Furthermore, 
without loss of generality, let us assume $\lambda_1$ in (\ref{eq:rational}) is the dominant pole of system (\ref{eq:state_space}). 
To simplify the notation, we set $\mu_\ell := \lambda^{(1)}_\ell$ for the sequence $\{ \lambda^{(1)}_\ell \}$  
generated by the algorithm, 
and investigate how small $| \mu_{j+1} - \lambda_1 |$ as compared to $| \mu_{j} - \lambda_1 |$ assuming
$\mu_{j}$ and $\mu_{j+1}$ are sufficiently close to $\lambda_1$.

Our analysis operates on the function
\begin{equation}\label{eq:defn_f}
	f(s)	\;\;	:=	\;\;	
	\frac{  | \det (sE - A) |^2 }{ \| C \cdot \text{adj} (sE - A) \cdot B \|_F^2}
\end{equation}
and its reduced counterpart at the end of the $j$th subspace iteration, that is
\begin{equation}\label{eq:defn_fr}
	f_j(s)	\;\;	:=	\;\;	
	\frac{ | \det (sW_j^\ast E V_j - W_j^\ast A V_j) |^2 }{ \| C V_j \cdot \text{adj} (s W_j^\ast E V_j - W_j^\ast A V_j) \cdot  W_j^\ast B \|_F^2},
\end{equation}
where $\text{adj}(\cdot)$ denotes the adjugate of its matrix argument. Clearly, using the notation $\Lambda(F,G)$
for the set of finite eigenvalues of the pencil $L(s) = F - s G$, we have
\begin{equation*}
	\begin{split}
	&	f(s)
			=	
	\frac{1}{\| H(s) \|_F^2} \;\;	 \forall s \not\in \Lambda(A,E),
	\quad	
	f_j(s)
			=
	\frac{1}{\| H^{{\mathcal W}_j , {\mathcal V}_j}_\red(s) \|_F^2} \;\;	 \forall s \not\in \Lambda(W_j^\ast A V_j , W_j^\ast E V_j).
	\end{split}
\end{equation*}

Moreover, $f(s)$ and $f_j(s)$ are well-defined under mild assumptions even at the finite eigenvalues of the associated pencils.
For instance, as we show next, unless the left eigenspace associated with $\lambda_1$ is orthogonal to $\text{Ran}(B)$ (i.e., 
unless $\lambda_1$ is uncontrollable), or the right eigenspace associated with $\lambda_1$ is orthogonal to $\text{Ran}(C^\ast)$ 
(i.e., unless $\lambda_1$ is unobservable), $f(s)$ is well-defined at $\lambda_1$. To see this, let us consider 
\begin{equation}\label{eq:adj_function}
	\text{adj}(s E - A)	\;\;	=	\;\;		
			\text{det}(s E - A) (s E  -  A)^{-1}	
\end{equation}
near $\lambda_1$. Observe that $\text{adj}(s E - A)$ is continuous everywhere, in particular at $s = \lambda_1$. Hence, by taking the limits
of both sides in (\ref{eq:adj_function}) as $s \rightarrow \lambda_1$ and recalling the Kronecker canonical form (\ref{eq:Kronecker}), 
it follows that
\begin{equation*}
	C \: \text{adj}(\lambda_1 E - A)  \: B		\;\;		=	\;\;	
			 	(-1)^{n-\widetilde{n}} \cdot \text{det}(W^{-\ast} V^{-1}) \cdot  CV \: 
						\left[
							\begin{array}{cc}
									\prod_{j = 2}^{\widetilde{n}} \lambda_1 - \lambda_j	&		0		\\
											0	&	0		\\
							\end{array}
						\right] \: W^\ast B.
\end{equation*}
From here we deduce that, unless $w_1 \bot \text{Ran}(B)$ or $v_1 \bot \text{Ran}(C^\ast)$, the denominator in (\ref{eq:defn_f}) is nonzero
at $s = \lambda_1$. Similarly, if $\mu_{j+1}$ is a controllable and observable pole of the reduced
transfer function $H^{{\mathcal W}_j , {\mathcal V}_j}_\red$, then $f_j(s)$ is well-defined at $\mu_{j+1}$.  
To summarize, simple conditions, such as the minimality of the original system and the reduced system at the end
of the $j$th subspace iteration (as the minimality implies both the controllability and the observability of all poles), 
guarantee the well-posedness of $f(s)$ and $f_j(s)$ everywhere. We make the 
following assumptions that guarantee the well-posedness of $f(s)$ and $f_j(s)$ at $\lambda_1$ and $\mu_{j+1}$
throughout the rest of this section.
\begin{assumption}\label{assum:well-posed}
The following conditions are satisfied:
\begin{enumerate}
	\item	$C \cdot \mathrm{adj} (\lambda_1 E - A) \cdot B \neq 0$.
	\item For prescribed $\beta > 0$, we have $\| C V_j \cdot \mathrm{adj} (\mu_{j+1} W_j^\ast E V_j - W_j^\ast A V_j) \cdot  W_j^\ast B \|_2 \geq \beta$.
\end{enumerate}
\end{assumption} 


It is evident that the numerators and denominators in (\ref{eq:defn_f}) and (\ref{eq:defn_fr}) defining $f$ and $f_j$ are
polynomials in $\Re s$ and $\Im s$, the real and imaginary parts of $s$. 
It can be shown by also employing Assumption \ref{assum:well-posed} that these functions are three-times differentiable with 
respect to $\Re s$, $\Im s$ with all of their first  three derivatives bounded in a ball $B(\lambda_1, \delta) := \{ z \in {\mathbb C} \: | \: | z - \lambda_1 | \leq \delta \}$ 
that contain $\mu_j$, $\mu_{j+1}$ (recall that $\mu_j$, $\mu_{j+1}$ are assumed to be sufficiently close to $\lambda_1$). 
Indeed, it can be shown that the radius $\delta$ and bounds on the derivatives of $f_j$ are uniform over all 
$V_j, W_j$ as long as the second condition in Assumption \ref{assum:well-posed} is met. The arguments similar to those in \cite[Section 3]{Aliyev2017}
can be used to show the existence of the ball, and the uniformity of the bounds.

The interpolatory properties between $f(s)$ and $f_j(s)$ and their first two derivatives at $\mu_j$ follow from 
Lemma \ref{thm:main_interpolation}. In the subsequent arguments, we use the notations
\begin{equation*}
	\begin{split}
	&
		\hskip 10ex
	f'(\widehat{s})
		\;	=	\;
	\left[
		\begin{array}{cc}
			\frac{\partial f}{\partial \Re s} (\widehat{s})
				&
			\frac{\partial f}{\partial \Im s} (\widehat{s})
		\end{array}
	\right]^T,	\quad
	f'_j(\widehat{s})
		\;	=	\;
	\left[
		\begin{array}{cc}
			\frac{\partial f_j}{\partial \Re s} (\widehat{s})
				&
			\frac{\partial f_j}{\partial \Im s} (\widehat{s})
		\end{array}
	\right]^T			\\[.5em]
	&
	\nabla^2 f(\widehat{s})
		\;	=	\;
	\left[
		\begin{array}{cc}
			\frac{\partial^2 f}{\partial \Re s \: \partial \Re s} (\widehat{s})
				&
			\frac{\partial^2 f}{\partial \Re s \: \partial \Im s} (\widehat{s})		\\[.3em]
			\frac{\partial^2 f}{\partial \Im s  \: \partial \Re s} (\widehat{s})
				&
			\frac{\partial^2 f}{\partial \Im s  \: \partial \Im s} (\widehat{s})
		\end{array}
	\right],	\quad
	\nabla^2 f_j(\widehat{s})
		\;	=	\;
	\left[
		\begin{array}{cc}
			\frac{\partial^2 f_j}{\partial \Re s \: \partial \Re s} (\widehat{s})
				&
			\frac{\partial^2 f_j}{\partial \Re s \: \partial \Im s} (\widehat{s})		\\[.3em]
			\frac{\partial^2 f_j}{\partial \Im s \: \partial \Re s} (\widehat{s})
				&
			\frac{\partial^2 f_j}{\partial \Im s \: \partial \Im s} (\widehat{s})
		\end{array}
	\right]
	\end{split}
\end{equation*}
at a given $\widehat{s} \in {\mathbb C}$, where $f$ and $f_j$ are twice differentiable with respect to $\Re s$, $\Im s$.
Moreover, ${\mathcal R} : {\mathbb C} \rightarrow {\mathbb R}^2$ is the linear map 
$
   {\mathcal R}(z)
   	\:	:=	\:
	\left[
		\begin{array}{cc}
			\Re z
				&
			\Im z
		\end{array}
	\right]^T
$.
\begin{theorem}\label{thm:intermed_inter}
Suppose $\mu_j$ is not an eigenvalue of $L(s) = A - sE$ and not a pole of $H^{{\mathcal W}_j, {\mathcal V}_j}_\red$,
and Assumption \ref{assum:well-posed} holds. Then
\begin{equation}\label{eq:interpol_prop}
	f'(\mu_j) \; = \; f_j'(\mu_j)	\quad\quad	\text{and}	\quad\quad	\nabla^2 f(\mu_j)	\;	=	\;	\nabla^2 f_j(\mu_j).
\end{equation}
\end{theorem}
\begin{proof}
We deduce from Lemma \ref{thm:main_interpolation} that 
\begin{equation}\label{eq:transfer_inter}
	H(\mu_j) = H^{{\mathcal W}_j, {\mathcal V}_j}_\red(\mu_j),	\;\;	H'(\mu_j) = [H^{{\mathcal W}_j, {\mathcal V}_j}_\red]'(\mu_j),	\;\;
	H''(\mu_j) = [H^{{\mathcal W}_j, {\mathcal V}_j}_\red]''(\mu_j).
\end{equation}
The result is obtained by differentiating
\[
	f(s)	\;	=	\;	\frac{1}{ \text{Trace}( H(s)^\ast H(s) ) },	\quad
	f_j(s) 	\;	=	\;  \frac{1}{\text{Trace}( H^{{\mathcal W}_j, {\mathcal V}_j}_\red(s)^\ast H^{{\mathcal W}_j, {\mathcal V}_j}_\red(s)) }
\]
and employing the equalities in (\ref{eq:transfer_inter}).
\end{proof}

We have $f(s) \geq 0$ and $f_j(s) \geq 0$ for all $s \in B(\lambda_1, \delta)$ by the defining equations in (\ref{eq:defn_f}) 
and (\ref{eq:defn_fr}), as well as $f(\lambda_1) = f_j(\mu_{j+1}) = 0$. Consequently, $\lambda_1$, $\mu_{j+1}$ are
global minimizers of $f$, $f_j$ implying
$
			f'(\lambda_1) 	\:	=	\:	 f'_j(\mu_{j+1}) \:	=	\: 0.
$
It follows that
\begin{equation*}
	\begin{split}
	0	\;	=	\;	f'(\lambda_1)	\;	&	=	\;	
		f'(\mu_j)	
				+	 \int_0^1	\nabla^2 f(\mu_j + t (\lambda_1 - \mu_j))  {\mathcal R}(\lambda_1 - \mu_j) \: {\rm d} t		\\[.5em]
			\;	&	=	\;
		f'(\mu_j)		+	\nabla^2 f(\mu_j) {\mathcal R}(\lambda_1 - \mu_j) 	\\
		&	\hskip 13ex
				 + \int_0^1	\left\{ \nabla^2 f(\mu_j + t (\lambda_1 - \mu_j))  -  \nabla^2 f(\mu_j)  \right\} {\mathcal R}(\lambda_1 - \mu_j) \: {\rm d} t	\\[.5em]
		\;	&	=	\;		
		f_j'(\mu_j)		+	\nabla^2 f_j(\mu_j) {\mathcal R}(\lambda_1 - \mu_j) 		\\
		&	\hskip 13ex
				 + \int_0^1	\left\{ \nabla^2 f(\mu_j + t (\lambda_1 - \mu_j))  -  \nabla^2 f(\mu_j)  \right\} {\mathcal R}(\lambda_1 - \mu_j) \: {\rm d} t,
	\end{split}
\end{equation*}
where the last equality is due to (\ref{eq:interpol_prop}).
An application of Taylor's theorem with second order remainder to $f_j'(s)$ at $\mu_{j+1}$ about $\mu_j$ 
together with $f'_j(\mu_{j+1}) = 0$ yield
\begin{equation*}
	\begin{split}
		f_j'(\mu_j)		+	\nabla^2 f_j(\mu_j) {\mathcal R}(\lambda_1 - \mu_j)
		\;\;	&  =	\;\;
		\nabla^2 f_j(\mu_j) {\mathcal R}(\lambda_1 - \mu_{j+1}) 	+	{\mathcal O}(  \left| \mu_{j+1} - \mu_j \right|^2 )	\\
			&	=	\;\;\:
		\nabla^2 f(\mu_j) {\mathcal R}(\lambda_1 - \mu_{j+1}) 	+	{\mathcal O}(  \left| \mu_{j+1} - \mu_j \right|^2 ).
	\end{split}
\end{equation*}	
Substituting the right-hand side of the last equality above in the previous equation, we deduce
\begin{equation}\label{eq:intermed}
	\begin{split}
	\small	\nabla^2 f(\mu_j) {\mathcal R}(\mu_{j+1} - \lambda_1)		&	\small =			
					\int_0^1	\left\{ \nabla^2 f(\mu_j + t (\lambda_1 - \mu_j))  -  \nabla^2 f(\mu_j)  \right\} {\mathcal R}(\lambda_1 - \mu_j)  {\rm d} t  \\
					& 	\small    \hskip 36ex		\: + \:   {\mathcal O}(  \left| \mu_{j+1} - \mu_j \right|^2 ).
	\end{split}
\end{equation}

Let us also suppose $\nabla^2 f(\lambda_1)$ is invertible. Then, by the continuity of $\nabla^2 f$ in $B(\lambda_1, \delta)$
and letting $\sigma_{\min} (\nabla^2 f(s))$ denote the smallest singular value of $\nabla^2 f(s)$, 
there exists a constant $\eta > 0$ such that 
\begin{equation}\label{eq:intermed2}
		\sigma_{\min} (\nabla^2 f(s)) \geq \eta  \quad\quad   \forall s \in B(\lambda_1, \delta), 
\end{equation}
by choosing $\delta$ even smaller if necessary. Furthermore, boundedness of the third derivatives of $f$ implies the Lipschitz continuity of its second
derivatives, in particular the existence of a constant $\gamma > 0$ such that
\begin{equation}\label{eq:intermed3}
		\| \nabla^2 f(\mu_j + t (\lambda_1 - \mu_j))  -  \nabla^2 f(\mu_j) \|_2 \leq \gamma t | \lambda_1 - \mu_j |  	\quad\quad	\forall t \in [0,1].
\end{equation}
By taking the 2-norms of both sides in (\ref{eq:intermed}), using the triangle inequality, as well as the inequalities
(\ref{eq:intermed2}), (\ref{eq:intermed3}), we obtain
\[
		\eta \left| \lambda_1 - \mu_{j+1} \right|  \;\; \leq \;\;  (\gamma/2)  \left| \lambda_1 - \mu_j \right|^2 \: + \: {\mathcal O}(  \left| \mu_{j+1} - \mu_j \right|^2 ).
\]
Finally, noting that ${\mathcal O}(  \left| \mu_{j+1} - \mu_j \right|^2 )$ terms are bounded from above by 
$c \left| \mu_{j+1} - \mu_j \right|^2 \leq 2 c \{  \left| \lambda_1 - \mu_j \right|^2  +  \left| \lambda_1 - \mu_{j+1} \right|^2 \}
\leq  2 c \{  \left| \lambda_1 - \mu_j \right|^2  +  \delta \left| \lambda_1 - \mu_{j+1} \right| \}$ for some constant $c$, we have
\[
		\hskip -4ex
		 (\eta - 2c\delta)  \left| \lambda_1 - \mu_{j+1} \right|  \;\; \leq \;\;  (\gamma/2 + 2c)  \left| \lambda_1 - \mu_j \right|^2.
\]
Our findings are summarized in the next result.
\begin{theorem}[At Least Quadratic Convergence of the Subspace Framework]\label{thm:quick_convergence}
Assuming that the iterates $\mu_j, \mu_{j+1}$ of Algorithm \ref{alg:SM} are sufficiently close to $\lambda_1$, the Hessian $\nabla^2 f(\lambda_1)$ is invertible,
and the assumptions of Theorem \ref{thm:intermed_inter} hold, we have
\[
			\left| \lambda_1 - \mu_{j+1} \right|  \;\; \leq \;\;  C  \left| \lambda_1 - \mu_j \right|^2
\]
for some constant $C > 0$.
\end{theorem}


\section{Practical Details}\label{sec:prac_details}
Here, we spell out some details that one has to take into consideration
in a practical implementation of Algorithm \ref{alg:SM}.

\subsection{Forming Initial Subspaces}\label{sec:init_subspaces}
Initial subspaces in line \ref{init_subspaces} of Algorithm \ref{alg:SM} are constructed so as to attain
Hermite interpolation between the full and reduced system at prescribed points $\rho_1, \dots, \rho_\varphi \in {\mathbb C}$.
It is desirable that $\rho_1, \dots, \rho_\varphi$ are not very far away from the dominant poles of $H$.

One possibility for the selection of these initial interpolation points is to form a rational approximation for
the full transfer function $H$, then use the dominant poles of the rational approximation. The AAA algorithm \cite{Nakatsukasa2018} 
is widely employed at the moment for such rational approximation problems. We have attempted to approximate
the entries of $H$ using the AAA algorithm, but, on benchmark examples, this approach usually did not yield points close to the dominant poles.

Instead, we apply the algorithm in \cite{Aliyev2017} for ${\mathcal L}_\infty$-norm computation crudely, which gives rise to
a reduced order model that approximates $H$ reasonably well on some parts of the imaginary axis where it exhibits large peaks.
The initial interpolation points $\rho_1, \dots, \rho_\varphi$ are then set equal to the dominant poles of the retrieved
reduced order model. 

\subsection{Termination Condition}
A natural choice for a termination condition in line \ref{return_spec} of Algorithm \ref{alg:SM} is based on the $\infty$-norms 
of the residuals
\begin{equation}\label{eq:res}
	\mathsf{R}\mathsf{s}( \lambda^{(j)}_\ell , v^{(j)}_\ell )   \:   :=   \:   \|  (A - \lambda^{(j)}_\ell E) \cdot V_{\ell - 1} v^{(j)}_\ell \|_\infty
\end{equation}
for $j = 1,\dots, \kappa$. If all of $\mathsf{R}\mathsf{s}( \lambda^{(j)}_\ell , v^{(j)}_\ell )$ are smaller than a prescribed tolerance \texttt{tol}, then we terminate
with $\zeta_j = \lambda^{(j)}_\ell$, $z_j =  V_{\ell - 1} v^{(j)}_\ell$ for $j = 1,\dots, \kappa$.


\subsection{Deciding the Convergence of a Dominant Pole Estimate}
The decision whether the estimates $\lambda^{(j)}_\ell , v^{(j)}_\ell$ have converged or not in line \ref{check_convergence}
of Algorithm \ref{alg:SM} is also given based on the residual $\mathsf{R}\mathsf{s}( \lambda^{(j)}_\ell , v^{(j)}_\ell )$ in (\ref{eq:res}).
Specifically, if $\mathsf{R}\mathsf{s}( \lambda^{(j)}_\ell , v^{(j)}_\ell ) < \texttt{tol}$, then $\lambda^{(j)}_\ell , v^{(j)}_\ell$
are deemed as already converged, and, as a result, the subspace expansion to Hermite interpolate at $\lambda^{(j)}_\ell$ is skipped.
The tolerance \texttt{tol} used here is \\[.2em]
the same as the one used for termination above.

\subsection{Solutions of Linear Systems}
The main computational burden of Algorithm \ref{alg:SM} is due to the solutions of linear
systems. Computation of the dominant poles for the reduced problems in line \ref{inter_point} by the QZ 
algorithm and orthonormalization of the bases for the projection subspaces in line \ref{orthogonalize} are quite 
negligible compared to the solutions of linear systems.

\smallskip

At each subspace iteration, several linear systems need to be solved in lines \ref{sub_expand_st} - \ref{sub_expand_end} \\[.15em]  
of Algorithm \ref{alg:SM}. In this part of the algorithm and inside the for loop over $j$, the coefficient matrix for all of the linear systems  
is either $(A - \lambda^{(j)}_\ell E)$ or $(A - \lambda^{(j)}_\ell E)^\ast$. We compute an LU factorization of $(A - \lambda^{(j)}_\ell E)$ 
only once for solving all of these linear \\[.15em]
systems. Then the resulting triangular linear systems are solved. 

\subsection{Orthonormalization of the Bases for the Subspaces}\label{sec:orthonormalize}
As $\lambda^{(j)}_\ell$ tends to converge with respect to $\ell$, the new directions $\widetilde{V}, \widetilde{W}$ to be included in the subspaces 
${\mathcal V}_{\ell-1}, {\mathcal W}_{\ell-1}$ nearly lie in these subspaces. Hence, without orthonormalization, the matrices
$[ V_{\ell - 1}  \; \widetilde{V} ]$ and  $[ W_{\ell - 1}  \; \widetilde{W} ]$ would be ill-conditioned, which in turn may result
in reduced order systems that are not computed accurately due to the rounding errors. Orthonormalization resolves
this numerical difficulty; in particular, the matrices $V_\ell, W_\ell$ with orthonormal columns at the end of the $\ell$th
iteration are well-conditioned.

In practice, in line \ref{orthogonalize} of Algorithm \ref{alg:SM}, 
when orthogonalizing the columns of $\widetilde{V}, \widetilde{W}$
with respect to the spaces spanned by the columns of $V_{\ell}, W_{\ell}$, we first apply
\begin{equation}\label{eq:orthogonalize}
	\widetilde{V}	\;	\longleftarrow  \; \widetilde{V}	\:  -  \:	V_{\ell} \, (V_{\ell}^\ast \widetilde{V} )
		\quad	\text{and}		\quad
	\widetilde{W}	\;	\longleftarrow  \; \widetilde{W}	 \: -  \:	W_{\ell} \, (W_{\ell}^\ast \widetilde{W})
\end{equation}
several times. Then the columns of $\widetilde{V}, \widetilde{W}$ are orthonormalized, and $V_{\ell}, W_{\ell}$ are 
augmented with these orthonormalized matrices.
The reorthogonalization strategy (i.e., application of (\ref{eq:orthogonalize}) several times)
improves the accuracy in the presence of rounding errors of the columns of $V_\ell, W_\ell$  
as orthonormal bases for the subspaces ${\mathcal V}_\ell, {\mathcal W}_\ell$.

\subsection{Complexity}{\label{sec:complexity}
Putting the formation of initial subspaces aside, there are three main computational tasks that
an actual implementation of Algorithm~\ref{alg:SM} must carry out at every subspace iteration
(i.e., per an iteration of the outer for loop).
\begin{enumerate}
	\item Computation of dominant poles of the reduced system in line \ref{inter_point}.
	\item Solutions of linear systems in lines \ref{exp_start}, \ref{exp_start2}, \ref{exp_start3}.
	\item Orthonormalizing the bases for projection subspaces in line \ref{orthogonalize}.
\end{enumerate}

As argued above the computational costs for items 1. and 3. are quite negligible. If the subspace
dimensions are $d$, which is typically much smaller than $n$, then 1. requires an application
of the QZ algorithm at a cost of ${\mathcal O}(d^3)$. On the other hand, 
the total cost due to 3. is ${\mathcal O} (n d^2)$ for each one of the at most
$\kappa$ dominant poles, as line \ref{orthogonalize} involves two orthonormalizations with respect 
to $n\times d$ matrices with orthonormal columns. The constant hidden in this last ${\mathcal O}(\cdot)$ notation
is small.

As for item 2., at the $\ell$th subspace iteration, for every dominant pole estimate $\mu$,
an LU factorization of the $n\times n$ matrix $A - \mu E$ is computed once, and $\min\{ m , p \} (q+1)$
back and forward substitutions are applied to $n\times n$ triangular systems. Hence, taking all dominant 
pole estimates into account, at most $\kappa$ LU factorizations,  $\kappa \min \{ m , p \} (q+1)$ forward 
and back substitutions on systems of size $n\times n$ need to be performed per subspace iteration.
These tasks related to 2. typically dominate the computations, and determine the runtime. 
If the system at hand is such that LU factorizations, back and forward substitutions can be carried out
in linear time, say at a cost of $cn$ for a small constant $c$, then orthonormalization costs may also
be influential, but most often this is not the case. Usually, LU factorizations are the
most dominant factor. The number of forward  and back substitutions increases
linearly with respect to $m$ and $p$, so, for a system with many inputs and outputs, time required
for forward and back substitutions may be significant, perhaps as significant as LU factorizations.
The number of forward and back substitutions increases also linearly with respect to $q$, the parameter 
that determines the number of derivatives to be interpolated, yet in practice we choose $q$ small 
(e.g., $q = 1$), as this already ensures quadratic convergence.


\section{Numerical Results}\label{sec:NR}
We have implemented Algorithm \ref{alg:SM} in Matlab taking the practical issues in Section \ref{sec:prac_details}
into account. Here, we perform numerical experiments with it in Matlab 2020b on an an iMac with Mac OS~12.1 
operating system, Intel\textsuperscript{\textregistered} Core\textsuperscript{\texttrademark} i5-9600K CPU and 32GB RAM.

Throughout, our implementation terminates
when $\mathsf{R}\mathsf{s}( \lambda^{(j)}_\ell , v^{(j)}_\ell ) < \texttt{tol}$
for $j = 1,\dots,\kappa\,$ for the tolerance $\texttt{tol} = 10^{-7}$ unless otherwise specified, where $\kappa$ is 
the number of dominant poles prescribed,  and the residual $\mathsf{R}\mathsf{s}( \lambda^{(j)}_\ell , v^{(j)}_\ell )$ is 
as in (\ref{eq:res}). 
As discussed in Section \ref{sec:init_subspaces}, the initial subspaces ${\mathcal V}_0$, ${\mathcal W}_0$ 
are chosen so that Hermite interpolation is attained at the ten most dominant poles of a crude reduced model 
obtained by applying the algorithm in \cite{Aliyev2017}, excluding Section \ref{sec:ndpsdim} where
we investigate the effect of the initial subspace dimension on the runtime. The interpolation parameter in Algorithm \ref{alg:SM}
and Theorem \ref{thm:main_interpolation} that determines how many derivatives of the reduced transfer
function and full transfer function match is set equal to $q = 1$ in all of the experiments.
We follow this practice even for the rectangular transfer functions with $m \neq p$.
The rapid convergence result and its derivation in Section \ref{sec:Rconv} apply to rectangular 
transfer functions when $q \geq 2$. Yet, we observe a quick convergence in the rectangular
setting on benchmark examples even with $q = 1$.

We assume throughout that the descriptor system at hand as in (\ref{eq:state_space}) has real coefficient 
matrices $A, E, B, C, D$. Under this assumption, the dominant poles come in conjugate pairs, i.e., if $z$ is 
one of the poles, unless $z$ is real, its conjugate $\overline{z}$ is also a pole with exactly the same dominance metric
as for $z$. Our implementation computes only the dominant poles with nonpositive imaginary parts;
in particular, we extract the dominant poles of the reduced problem with nonpositive imaginary parts and interpolate 
only at these poles at every subspace iteration. In the subsequent subsections of this section, what we refer as the most dominant 
$\kappa$ poles are indeed the most dominant $\kappa$ poles among the poles with nonpositive imaginary parts.
As a result, we count a conjugate pair of dominant poles only once and not twice. For instance, when we report the 
five most dominant poles and if those poles are not real, in reality the ten most dominant poles are computed.

In the next three subsections, we report the results by our implementation of Algorithm \ref{alg:SM} on benchmark examples 
for dominant pole estimation made available by Joost Rommes most of which can be accessed from
his website\footnote[2]{\url{http://sites.google.com/site/rommes/software}}, as well as benchmark examples from SLICOT collection 
for model reduction\footnote[3]{\url{http://slicot.org/20-site/126-benchmark-examples-for-model-reduction}}. Comparisons are provided 
with the subspace accelerated multiple-input-multiple-output (MIMO) dominant pole algorithm (SAMDP) \cite{RM2006b},
specifically with the implementation on Rommes' website.

\subsection{Convergence and Illustration of the Algorithm on the \texttt{M80PI\_n} Example}
We first illustrate the convergence of Algorithm \ref{alg:SM} on the \texttt{M80PI\_n} example. 
This is an example with $n = 4182$ and $m = p = 3$. When we attempt to compute the most dominant
pole, it takes only two iterations to fulfill the termination condition, that is the residual of the most dominant
pole estimate and the corresponding eigenvector estimate is less than $10^{-7}$ after two subspace iterations.
The iterates generated and the corresponding residuals are given in Table \ref{table:M80_conv1}.
The progress in the iterates in the table is consistent with at-least-quadratic-convergence 
assertion of Theorem \ref{thm:quick_convergence}.

\begin{table}
\centering
\caption{The iterates and corresponding residuals of Algorithm \ref{alg:SM} to find the most
dominant pole of the \texttt{M80PI\_n} example.}
\label{table:M80_conv1}
\begin{tabular}{c|cl}
$\ell$  & $\lambda^{(1)}_\ell$ 	& 	 \phantom{aa} $\mathsf{R}\mathsf{s}( \lambda^{(1)}_\ell , v^{(1)}_\ell )$  \\    
\hline
1 & $-$\underline{2.88547}6680562e$+$01 $-$ \underline{5.0734196}87488e$+$03${\rm i}$  	&  	\phantom{aa} $8.398\cdot 10^{-4}$ \\
2 & $-$\underline{2.885470736631}e$+$01 $-$ \underline{5.073419627243}e$+$03${\rm i}$ 	 &  	\phantom{aa} $1.044\cdot 10^{-13}$ 
\end{tabular}
\end{table}

The convergence behavior is similar when we attempt to locate multiple dominant poles. In Table \ref{table:Multiple_res},
the residuals for the iterates of Algorithm \ref{alg:SM} are listed to estimate the most dominant five poles of the 
 \texttt{M80PI\_n} example. Only three subspace iterations suffice to satisfy the termination condition, that is
 all residual are below the tolerance $10^{-7}$ after three iterations. In the first iteration, Hermite interpolation is performed 
 at all of the five dominant pole estimates as the residuals are larger than $10^{-7}$. At the second iteration, the most
 dominant four pole estimates are deemed as already converged, since their residuals are below the convergence
 threshold, while Hermite interpolation is imposed at the estimate for the fifth dominant pole.  
 Notably at a dominant pole estimate where Hermite interpolation is imposed in Table \ref{table:Multiple_res},
 the corresponding residual decreases considerably. In general, we have observed a similar convergence behavior
on other examples. Only that it sometimes happens that not much progress is achieved towards convergence
in the first few iterations. But once a better global approximation of the transfer function is obtained, in particular
when the estimates for the dominant poles are close to the actual ones, convergence is very fast.

\begin{table} 
\caption{The residuals of the iterates of Algorithm \ref{alg:SM} to 
compute the most dominant five poles of the \texttt{M80PI\_n} example. Interpolation is performed 
at the iterates whose residuals are typed in blue italic.}
\label{table:Multiple_res}
\small
\begin{tabular}{c|ccccc}
\hskip 1ex $\ell$ \hskip 2ex  &  $\mathsf{R}\mathsf{s}\big(\lambda^{(1)}_\ell, v^{(1)}_\ell\big)$   &$\mathsf{R}\mathsf{s}\big(\lambda^{(2)}_\ell, v^{(2)}_\ell\big)$  & $\mathsf{R}\mathsf{s}\big(\lambda^{(3)}_\ell, v^{(3)}_\ell\big)$  &$\mathsf{R}\mathsf{s}\big(\lambda^{(4)}_\ell, v^{(4)}_\ell\big)$ & $\mathsf{R}\mathsf{s}\big(\lambda^{(5)}_\ell, v^{(5)}_\ell\big)$ \\    
\hline
\hskip 1ex 1 \hskip 2ex &  \textcolor{blue}{$\mathit{8.398 \cdot 10^{-4}}$}  &  \textcolor{blue}{$\mathit{3.051 \cdot 10^{-2}}$}  &   \textcolor{blue}{$\mathit{2.500 \cdot 10^{-1} \; \:}$} &  \textcolor{blue}{$\mathit{8.419 \cdot 10^{-1}}$}  &  \textcolor{blue}{$\mathit{1.651 \cdot 10^{-1}}$} \\
\hskip 1ex 2 \hskip 2ex &  $1.488 \cdot 10^{-13}$  &  $7.642 \cdot 10^{-14}$  &   $4.413 \cdot 10^{-9} \; \:$   &    $1.185 \cdot 10^{-8}$   &  \textcolor{blue}{$\mathit{9.539 \cdot 10^{-5}}$} \\
\hskip 1ex 3 \hskip 2ex &  $5.219 \cdot 10^{-14}$   & $3.553 \cdot 10^{-14}$  &   $2.324 \cdot 10^{-9} \; \;$ &  $6.616 \cdot 10^{-9}$  &  $5.182 \cdot 10^{-13}$
\end{tabular}
\end{table}

The convergence of Algorithm \ref{alg:SM} on the \texttt{M80PI\_n} example is also illustrated in Figure \ref{fig:M80n_conv}.
To illustrate the convergence better, here we start with a very crude reduced system initially; see the red dashed curve
in Figure \ref{fig:1a}. (A more accurate initial reduced model is used in Table \ref{table:M80_conv1} and \ref{table:Multiple_res},
as we normally apply the algorithm in \cite{Aliyev2017} to construct the initial reduced system a little more rigorously, 
a common practice in our implementation excluding this visualization.) The initial reduced transfer function 
$H^{{\mathcal W}_0 , {\mathcal V}_0}_\red$ in Figure \ref{fig:1a} interpolates the original transfer function 
$H$ at complex points, whose imaginary parts are marked with crosses on the horizontal axis.
In the same figure, the blue circles mark the imaginary parts of the five most dominant poles of $H^{{\mathcal W}_0 , {\mathcal V}_0}_\red$.
These five most dominant poles are used for interpolation next, giving rise to $H^{{\mathcal W}_1 , {\mathcal V}_1}_\red$, the 
reduced transfer function in Figure \ref{fig:1b}. Similarly, the dominant poles of $H^{{\mathcal W}_1 , {\mathcal V}_1}_\red$
whose imaginary parts are marked with blue circles in Figure \ref{fig:1b} are the next set of interpolation points, leading to the reduced transfer
function $H^{{\mathcal W}_2 , {\mathcal V}_2}_\red$ in Figure \ref{fig:1c}. The largest singular value of the transfer function $H^{{\mathcal W}_3, {\mathcal V}_3}_\red$
after three iterations in Figure \ref{fig:1d} seems to capture the largest singular value of the full transfer function 
over the imaginary axis already quite well, and (at least) the imaginary parts of the dominant poles of
$H^{{\mathcal W}_3, {\mathcal V}_3}_\red$ and $H$ appear to be close from the figure.

\begin{figure} 
 \centering
	\begin{tabular}{ll}
		\hskip -3ex
			 \subfigure[$\ell = 0$ \label{fig:1a}]{\includegraphics[width = .51\textwidth]{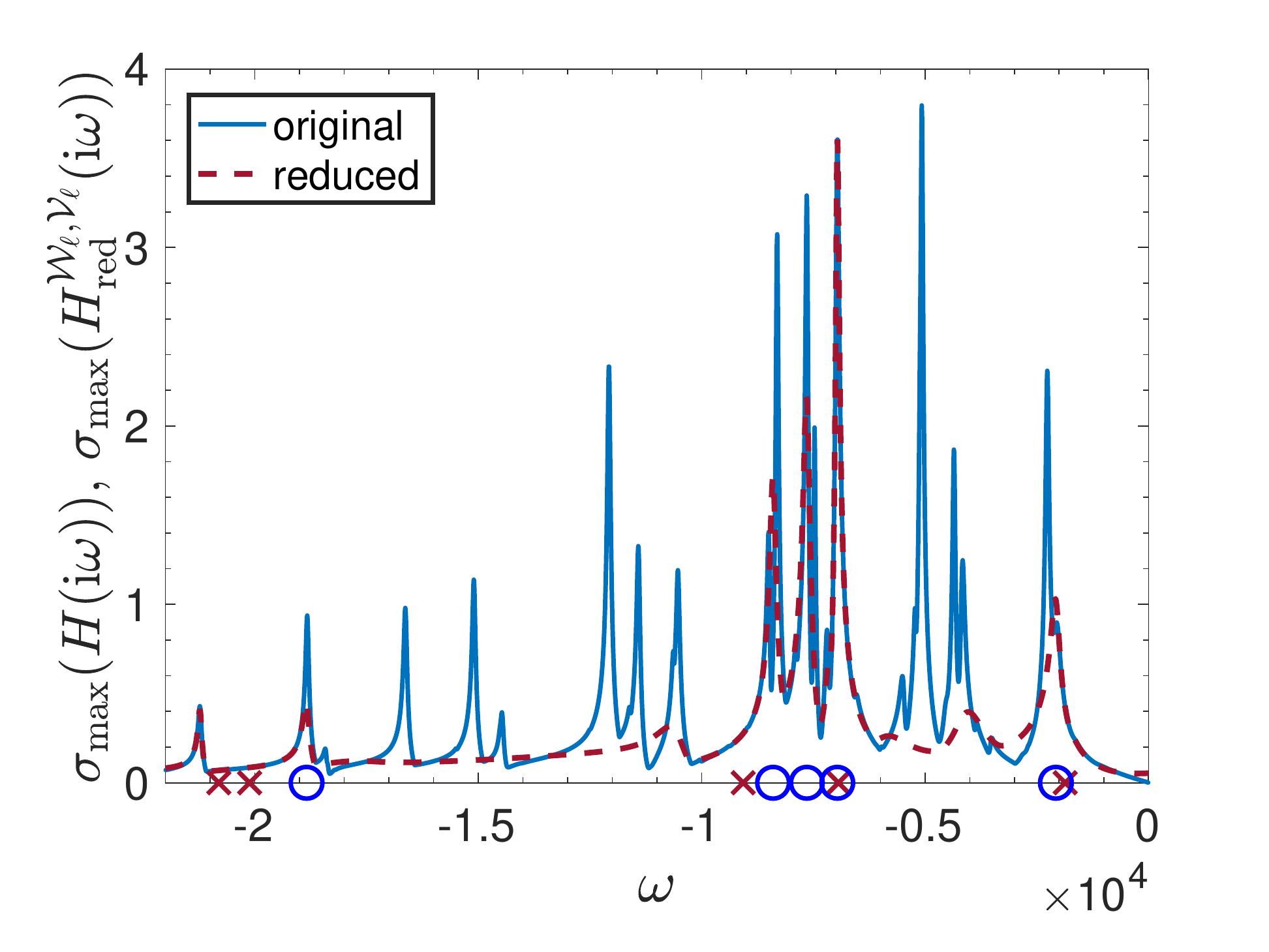}} & 
			 \subfigure[$\ell = 1$ \label{fig:1b}]{\includegraphics[width = .51\textwidth]{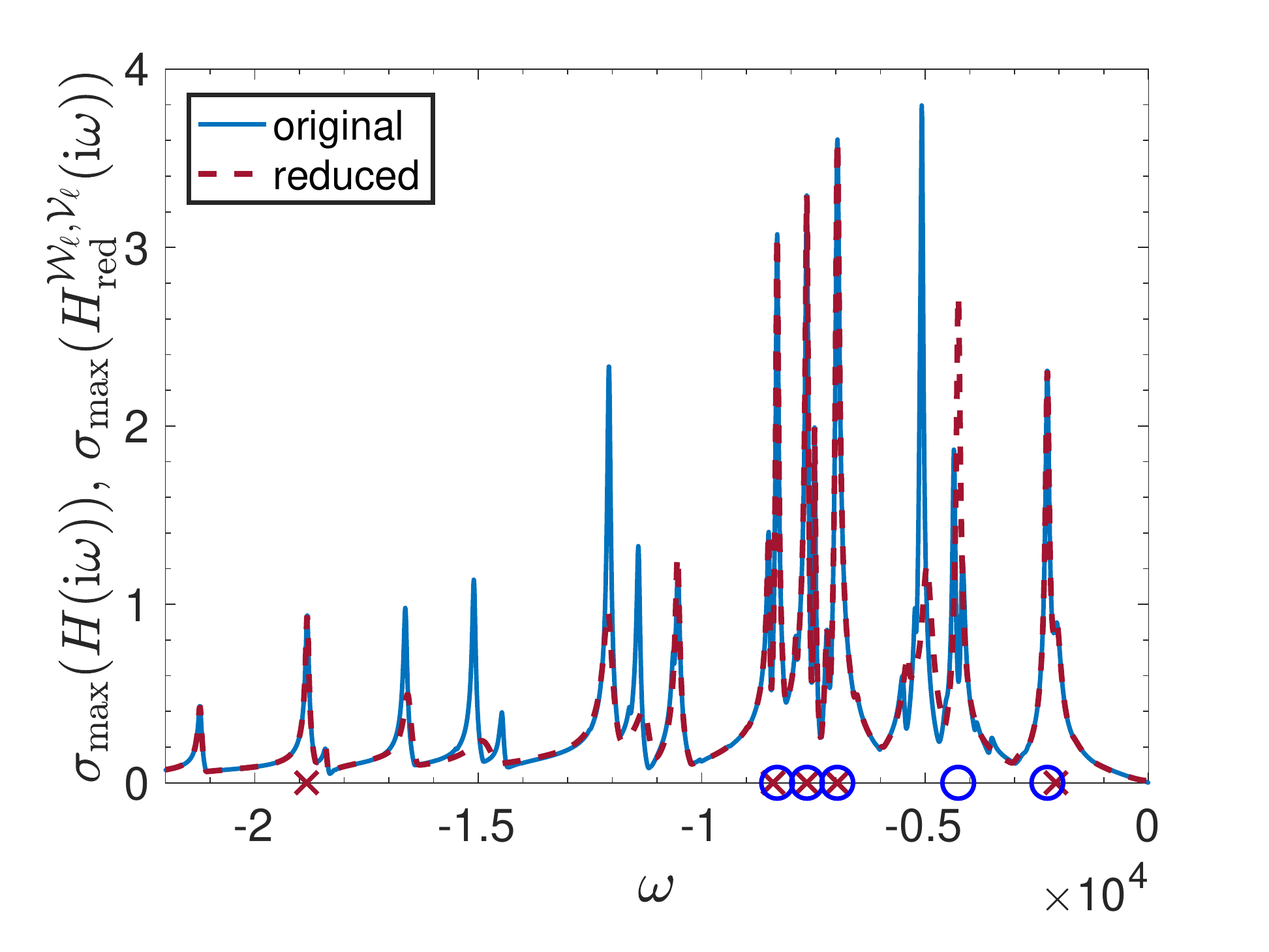}} \\  
		\hskip -3ex
			 \subfigure[$\ell = 2$ \label{fig:1c}]{\includegraphics[width = .51\textwidth]{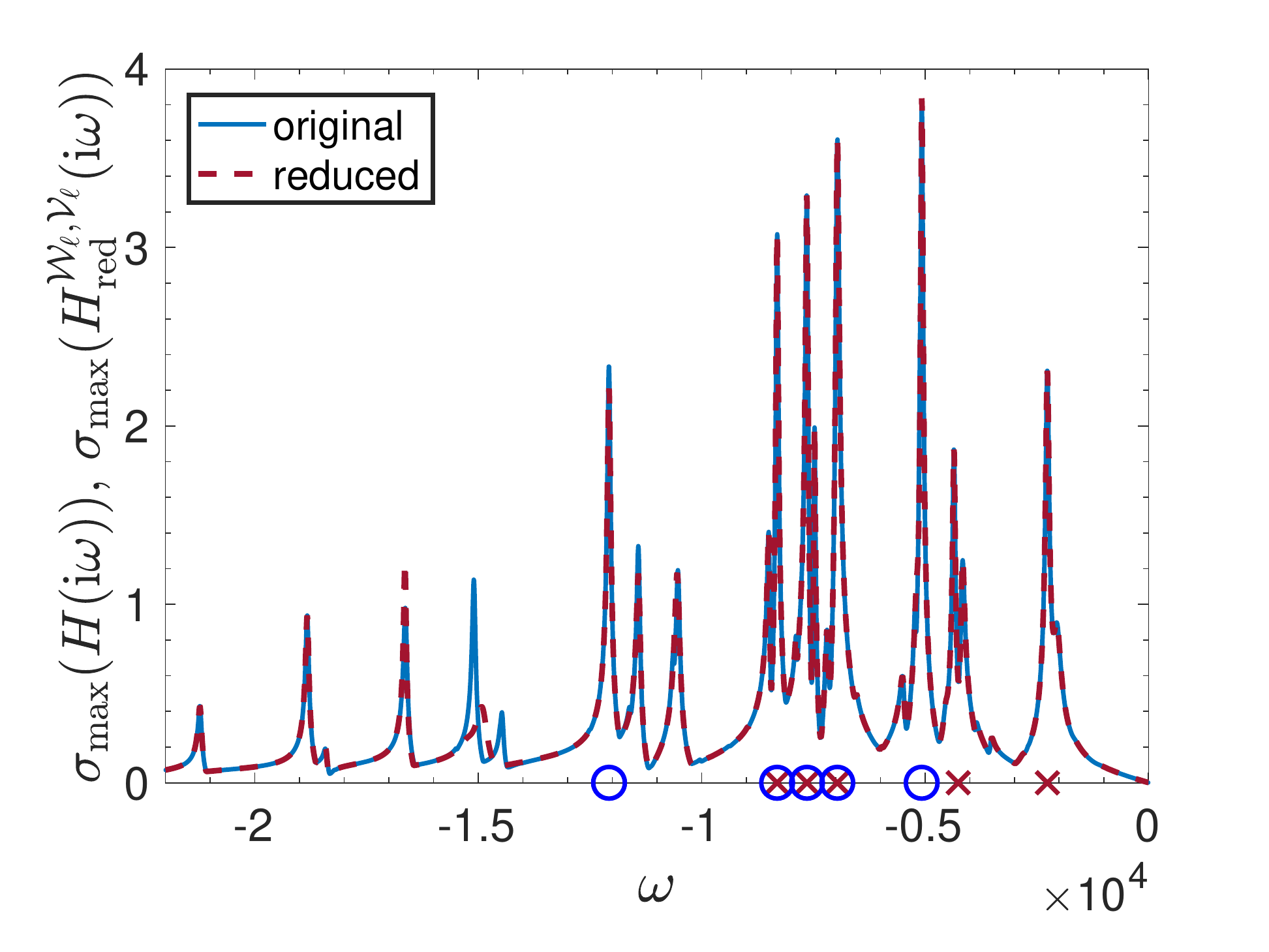}} &
			 \subfigure[$\ell = 3$ \label{fig:1d}]{\includegraphics[width = .51\textwidth]{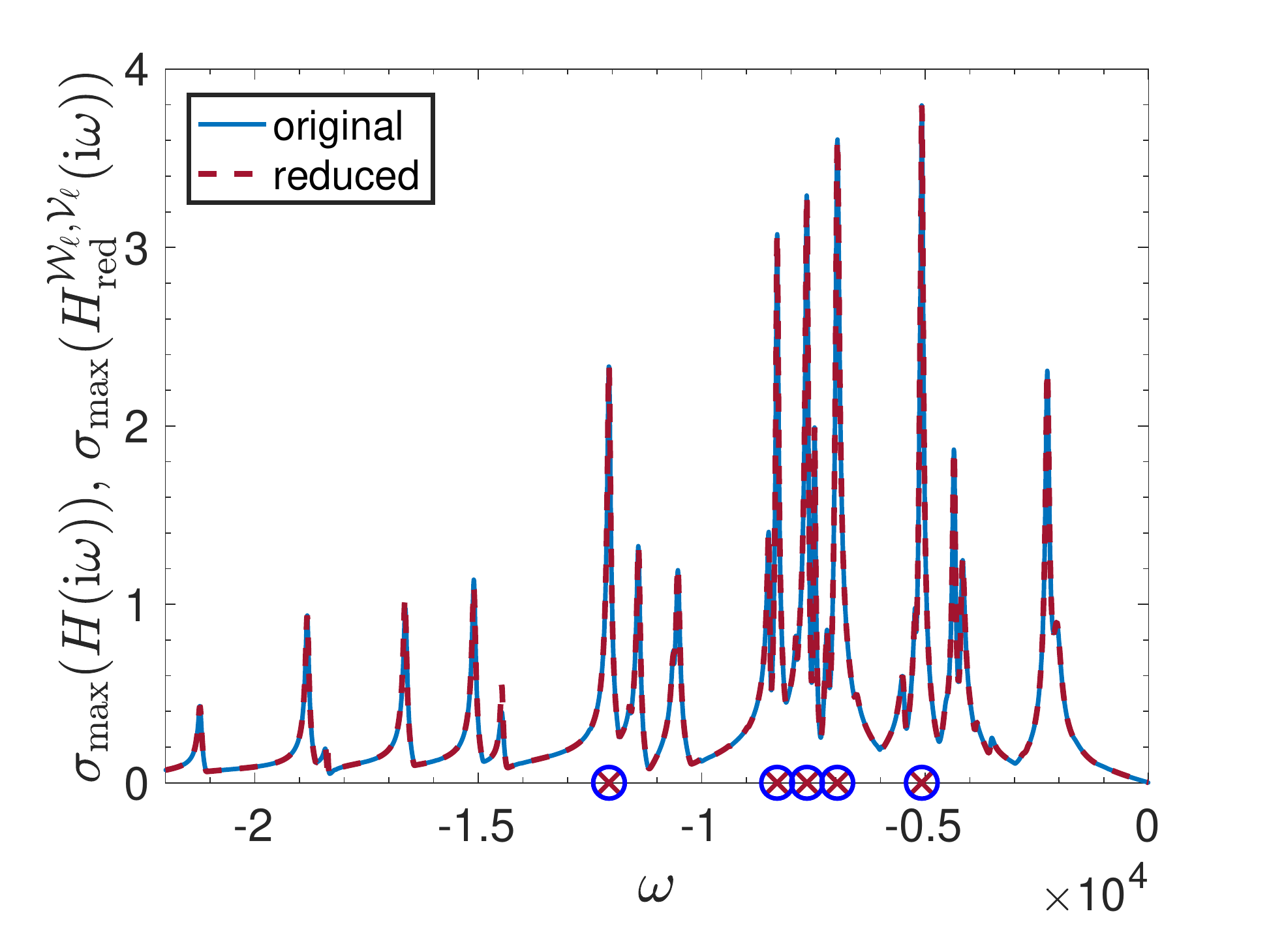}}	
		\end{tabular}   
		\caption{  
		The progress of Algorithm \ref{alg:SM} on the \texttt{M80PI\_n} example. The solid blue and red dashed curves
		are the plots of the largest singular values of $H({\rm i} \omega)$ and $H^{{\mathcal W}_\ell , {\mathcal V}_\ell}_\red({\rm i} \omega)$
		as functions of $\omega \in {\mathbb R}$. The red crosses and blue circles on the horizontal axis mark the imaginary parts
		of the interpolation points employed and dominant poles of  $H^{{\mathcal W}_\ell , {\mathcal V}_\ell}_\red$, respectively.
		}
		\label{fig:M80n_conv}
\end{figure}

\subsection{Results on Benchmark Examples}\label{sec:numexp_bench}
A comparison of our implementation of Algorithm \ref{alg:SM} and the subspace accelerated MIMO dominant 
pole algorithm (SAMDP) \cite{RM2006b} on 11 benchmark examples is provided in Table \ref{tab:results_bm}.

\textit{Parameter Values.}
Both algorithms are initialized with exactly the same ten points in the complex plane, namely the ten most
dominant poles of the reduced system produced by the algorithm in \cite{Aliyev2017}. As a result, the initial 
subspace dimension for Algorithm \ref{alg:SM} is $20 \cdot \min\{ m , p \}$, i.e., recall that $q = 1$ throughout,
which means that each interpolation point requires the inclusion of a subspace of dimension $2 \cdot \min \{ m, p \}$. 
We use the implementation of SAMDP with its default parameter settings with one exception, namely
the parameter that determines how LU factorizations should be computed. To be fair, both Algorithm \ref{alg:SM} 
and SAMDP employ the command \texttt{lu(A)}, when an LU factorization of $A$ needs to be computed.
Moreover, both algorithms solves the linear systems by exploiting the LU factorizations using the same sequence
of commands based on backslashes. The default minimum and maximum subspace dimensions for SAMDP
are 1 and 10, which we rely on in the numerical experiments. This means that SAMDP restarts the
projection subspaces in case they become of dimension greater than 10, and the restarted subspaces are of dimension 1.
In contrast, our implementation of Algorithm \ref{alg:SM} does not employ a restart strategy.
Our observation is that the performance of SAMDP remains nearly the same even without the restart strategy.

\textit{Comments on the Results.}
For all of these examples, we estimate the five most dominant poles of the full system using these two algorithms. For the first
two smaller problems (\texttt{CDplayer}, \texttt{iss}), the algorithms return exactly the same five dominant poles up to
prescribed tolerances. But for all other 9 examples, Algorithm \ref{alg:SM} returns consistently more
dominant poles compared to SAMDP; see the numbers inside the parentheses in the columns under ``Five Most
Dominant Poles'' in Table  \ref{tab:results_bm}, which represent the dominance metrics of the computed poles 
as in equation (\ref{eq:dp_criterion}). There are even examples for which three or four of the five
most dominant poles by Algorithm \ref{alg:SM} are more dominant than the ones computed by SAMDP; 
see, e.g., the \texttt{M10PI\_n} example for which only the most dominant poles are the same and
the remaining four differ in favor of Algorithm \ref{alg:SM} when their dominance metric is taken into consideration.
On smaller examples, specifically on the \texttt{CDplayer}, \texttt{iss}, \texttt{S40PI\_n}, \texttt{M010PI\_n} examples,
we have also computed all of the poles using the QR or QZ algorithm (i.e., using \texttt{eig} in Matlab), and verified
that the five most dominant poles by Algorithm \ref{alg:SM} listed in Table \ref{tab:results_bm} are indeed
the five most dominant poles up to the prescribed tolerances.

    %

%


\begin{table} 
 \caption{A comparison of Algorithm \ref{alg:SM} with SAMDP \cite{RM2006b} on 11 benchmark examples. 
 The numbers inside the parentheses in the columns under ``Five Most Dominant Poles'' are the dominance metrics of the computed poles as in equation (\ref{eq:dp_criterion}).
 Moreover, for each example, number of subspace iterations performed by Algorithm \ref{alg:SM} until termination is given in the column of $n_{\rm iter}$.} 
 \label{tab:results_bm}
\scriptsize
\hskip -8ex
 \begin{tabular}{c|c|c|c|r|r|cc}
\hskip -3ex  	& &   & & \multicolumn{2}{c}{Five Most Dominant Poles} & \multicolumn{2}{c}{time in s} \\[.3em]
\hskip -3ex   $\#$ & example & n,m,p  & $n_{\rm iter}$ & SAMDP \cite{RM2006b} \phantom{aaaa} & Alg.~\ref{alg:SM} \phantom{aaaaaaaaaa} & \cite{RM2006b}  & Alg.~\ref{alg:SM} \\
  \hline 
\hskip -3ex		&			&		&									&	1.0e$+$02 $\cdot$ \phantom{aaaaa(2.32e$+$06)}	&		1.0e$+$02 $\cdot$ \phantom{aaaaa(2.32e$+$06)}					&	&		\\[.2em]
\hskip -3ex  		&			&		&				&		 -0.0023 $\pm$ 0.2257{\rm i} (2.32e$+$06)	&	-0.0023 $\pm$ 0.2257{\rm i} (2.32e$+$06)		&			&		\\[.2ex]
\hskip -3ex &		&			&						&		-0.1227 $\pm$ 3.0654{\rm i} (3.36e$+$03)	&	 -0.1227 $\pm$ 3.0654{\rm i} (3.36e$+$03)		&		&		\\[.2ex]
\hskip -3ex 1 	&	 \texttt{CDplayer} 		&	120,2,2		&	1		&		-0.0781 $\pm$ 0.7775{\rm i} (5.56e$+$02)		&	-0.0781 $\pm$ 0.7775{\rm i} (5.56e$+$02)			&			0.04		&		0.01		\\[.2ex]
\hskip -3ex		&		&			&				&	 -0.1976 $\pm$ 1.9658{\rm i} (2.91e$+$02)		&	 -0.1976 $\pm$ 1.9658{\rm i} (2.91e$+$02)	&	&		\\[.2ex]
\hskip -3ex		&  		& 			&				 &  	 -0.0742 $\pm$ 0.7382{\rm i} (2.27e$+$02)	&   -0.0742 $\pm$ 0.7382{\rm i} (2.27e$+$02)     &      &   \\[.3em] 
\hline
\hskip -3ex  	&		&			&				&		-0.0039 $\pm$ 0.7751{\rm i} (1.16e$-$01)	&	-0.0039 $\pm$ 0.7751{\rm i} (1.16e$-$01)		&			&		\\[.2ex]
\hskip -3ex &		&			&				&		 -0.0100 $\pm$ 1.9920{\rm i} (3.38e$-$02)	&	 -0.0100 $\pm$ 1.9920{\rm i} (3.38e$-$02)		&		&		\\[.2ex]
\hskip -3ex 2 	&	 \texttt{iss} 		&	270,3,3		&	2		&		-0.0424 $\pm$ 8.4808{\rm i}  (1.20e$-$02)		&	-0.0424 $\pm$ 8.4808{\rm i} (1.20e$-$02)			&			0.02		&		0.02		\\[.2ex]
\hskip -3ex		&		&			&				&	 -0.1899 $\pm$ 37.9851{\rm i}  (1.07e$-$02)		&	-0.1899 $\pm$ 37.9851{\rm i} (1.07e$-$02)	&	&		\\[.2ex]
\hskip -3ex		&  		& 			&				 &  -0.0462 $\pm$ 9.2336{\rm i} (6.24e$-$03)	&    -0.0462 $\pm$ 9.2336{\rm i} (6.24e$-$03)     &      &   \\[.3em]
 \hline
\hskip -3ex	&		&		&									&	1.0e$+$04 $\cdot$ \phantom{aaaaa(3.29e$+$00)}	&		1.0e$+$04 $\cdot$ \phantom{aaaaa(3.29e$+$00)}					&	&		\\[.2em]
\hskip -3ex  	&		&		&									&	-0.0041 $\pm$ 0.6962{\rm i} (3.29e$+$00)			&	-0.0041 $\pm$ 0.6962{\rm i} (3.29e$+$00)	&	&		\\[.2ex]
\hskip -3ex 	&		&			&								&	 -0.0032 $\pm$ 0.7637{\rm i} (3.21e$+$00)	&	 -0.0032 $\pm$ 0.7637{\rm i} (3.21e$+$00)		&		&		\\[.2ex]
\hskip -3ex 3 	&	  \texttt{S40PI\_n}		&	2182,1,1		&	3		&		-0.0021 $\pm$ 0.7466{\rm i} (1.66e$+$00)		&	-0.0021 $\pm$ 0.7466{\rm i} (1.66e$+$00)			&			0.12		&		0.07		\\[.2ex]
\hskip -3ex		&		&			&							&	 -0.0036 $\pm$ 1.2053{\rm i} (1.26e$+$00)		&	 -0.0032 $\pm$ 0.4351{\rm i} (1.63e$+$00)	&	&		\\[.2ex]
\hskip -3ex		&  		& 			&							 &  	-0.0042 $\pm$ 1.0519{\rm i} (9.13e$-$01)	&   -0.0036 $\pm$ 1.2053{\rm i}  (1.26e$+$00)     &      &   \\[.3em]
\hline
\hskip -3ex	&		&		&									&	1.0e$+$04 $\cdot$ \phantom{aaaaa(3.31e$+$00)}	&		1.0e$+$04 $\cdot$ \phantom{aaaaa(3.31e$+$00)}					&	&		\\[.2em]
\hskip -3ex  	&		&		&									&	-0.0041 $\pm$ 0.6965{\rm i} (3.31e$+$00)		&	 -0.0041 $\pm$ 0.6965{\rm i} (3.31e$+$00)	&	&		\\[.2ex]
\hskip -3ex 	&		&			&								&	-0.0032 $\pm$ 0.7642{\rm i} (3.17e$+$00)		&	 -0.0032 $\pm$ 0.7642{\rm i} (3.17e$+$00)		&		&		\\[.2ex]
\hskip -3ex 4 	&	  \texttt{S80PI\_n}		&	4182,1,1		&	5		&	-0.0021 $\pm$ 0.7474{\rm i} (1.69e$+$00)		&	-0.0021 $\pm$ 0.7474{\rm i} (1.69e$+$00)			&			0.22		&		0.21		\\[.2ex]
\hskip -3ex		&		&			&							&	-0.0036 $\pm$ 1.2073{\rm i} (1.25e$+$00)		&	-0.0032 $\pm$ 0.4351{\rm i} (1.63e$+$00)				&	&		\\[.2ex]
\hskip -3ex		&  		& 			&							 &     -0.0043 $\pm$ 1.0532{\rm i}  (9.54e$-$01) 	&   -0.0036 $\pm$ 1.2073{\rm i}  (1.25e$+$00)    &      &   \\[.3em]
\hline
\hskip -3ex	&		&		&									&	1.0e$+$04 $\cdot$ \phantom{aaaaa(3.99e$+$00)}	&		1.0e$+$04 $\cdot$ \phantom{aaaaa(3.99e$+$00)}					&	&		\\[.2em]
\hskip -3ex  	&		&		&									&	-0.0028 $\pm$ 0.5035{\rm i} (3.99e$+$00)		&	-0.0028 $\pm$ 0.5035{\rm i} (3.99e$+$00)	&	&		\\[.2ex]
\hskip -3ex 	&		&			&								&	-0.0040 $\pm$ 0.2262{\rm i} (2.17e$+$00)		&	 -0.0032 $\pm$ 0.7530{\rm i} (3.98e$+$00)		&		&		\\[.2ex]
\hskip -3ex 5 	&	  \texttt{M10PI\_n}		&	682,3,3		&	5		&	-0.0032 $\pm$ 0.4340{\rm i} (1.86e$+$00)		&	-0.0028 $\pm$ 0.8126{\rm i} (2.92e$+$00)			&			0.07		&		0.08		\\[.2ex]
\hskip -3ex		&		&			&							&	-0.0043 $\pm$ 1.0316{\rm i} (1.12e$+$00)		&	-0.0041 $\pm$ 0.6905{\rm i} (2.90e$+$00)				&	&		\\[.2ex]
\hskip -3ex		&  		& 			&							 &    -0.0049 $\pm$ 0.4153{\rm i}  (9.59e$-$01) 		&     -0.0036 $\pm$ 1.1670{\rm i} (2.24e$+$00)    &      &   \\[.3em]
\hline
\hskip -3ex  	&		&		&									&	-0.8144 $\pm$ 5.3836{\rm i} (1.49e$+$02)		&	-0.8144 $\pm$ 5.3836{\rm i} (1.49e$+$02)	&	&		\\[.2ex]
\hskip -3ex 	&		&			&								&	-0.3471 $\pm$ 6.6647{\rm i} (3.25e$+$01)		&	-7.6146 $\pm$ 1.8625{\rm i} (5.75e$+$01)		&		&		\\[.2ex]
\hskip -3ex 6 	&	  \texttt{bips07\_1998}		&	15066,4,4		&	6		&	-1.0365 $\pm$ 6.3718{\rm i} (3.09e$+$01)		&	-8.9412 $\pm$ 0.2616{\rm i} (5.25e$+$01)			&			8.17		&		5.35		\\[.2ex]
\hskip -3ex		&		&			&							&	-0.7473 $\pm$ 6.2077{\rm i}  (1.84e$+$01)		&	-0.3471 $\pm$ 6.6647{\rm i} (3.25e$+$01)				&	&		\\[.2ex]
\hskip -3ex		&  		& 			&							 &  -0.6254 $\pm$ 6.4503{\rm i} (6.77e$+$00) 		&   -1.0365 $\pm$ 6.3718{\rm i} (3.09e$+$01)		   &      &   \\[.3em]
\hline
\hskip -3ex  	&		&		&									&	-0.7181 $\pm$ 5.3405{\rm i} (1.92e$+$02)		&	-0.7181 $\pm$ 5.3405{\rm i} (1.92e$+$02)	&	&		\\[.2ex]
\hskip -3ex 	&		&			&								&	-0.8114 $\pm$ 6.1083{\rm i} (5.77e$+$01)		&	-7.5029 $\pm$ 2.2357{\rm i} (7.33e$+$01)		&		&		\\[.2ex]
\hskip -3ex 7 	&	 \texttt{bips07\_3078}		&	21128,4,4		&	8		&	-0.8894 $\pm$ 6.1706{\rm i} (4.73e$+$01)		&	-0.8114 $\pm$ 6.1083{\rm i} (5.77e$+$01)			&			1.24		&		2.69		\\[.2ex]
\hskip -3ex		&		&			&							&	-0.9818 $\pm$ 6.3935{\rm i} (2.87e$+$01)		&	-6.0025 $\pm$ 0.0885{\rm i} (4.87e$+$01)				&	&		\\[.2ex]
\hskip -3ex		&  		& 			&							 & -0.5966 $\pm$ 6.2367{\rm i} (1.04e$+$01) 		&    -0.8894 $\pm$ 6.1706{\rm i} (4.73e$+$01)    &      &   \\[.3em]
\hline
\hskip -3ex  	&		&		&									&	-1.2798 $\pm$ 8.0934{\rm i} (4.52e$-$01)		&	-0.0153 $\pm$ 1.0916{\rm i} (4.04e$+$00)	&	&		\\[.2ex]
\hskip -3ex 	&		&			&								&	-1.1021 $\pm$ 7.9841{\rm i} (1.18e$-$01)		&	-1.2798 $\pm$ 8.0934{\rm i} (4.52e$-$01)		&		&		\\[.2ex]
\hskip -3ex 8 	&	\texttt{xingo\_afonso\_itaipu}		&	13250,1,1		&	7		&	-1.2713 $\pm$ 8.5866{\rm i} (7.94e$-$02)		&	-14.2871 $\pm$ 5.9445{\rm i} (2.03e$-$01)			&			0.67		&		0.87		\\[.2ex]
\hskip -3ex		&		&			&							&	-1.2180 $\pm$ 8.1776{\rm i} (3.73e$-$02)		&	-0.9002 $\pm$ 7.6356{\rm i} (1.36e$-$01)				&	&		\\[.2ex]
\hskip -3ex		&  		& 			&							 &  -1.4127 $\pm$ 8.0192{\rm i} (3.16e$-$02) 		&   -0.5987 $\pm$ 0.3306{\rm i} (1.31e$-$01)    &      &   \\[.3em]
\hline
\hskip -3ex  	&		&		&									&	-0.5208 $\pm$ 2.8814{\rm i} (1.45e$-$03)		&	-0.0335 $\pm$ 1.0787{\rm i} (2.76e$-$03)	&	&		\\[.2ex]
\hskip -3ex 	&		&			&								&	-0.5567 $\pm$ 3.6097{\rm i} (1.34e$-$03)		&	-0.5208 $\pm$ 2.8814{\rm i} (1.45e$-$03)		&		&		\\[.2ex]
\hskip -3ex 9 	&	 \texttt{ww\_vref\_6405}		&	13251,1,1		&	5		&	-0.1151 $\pm$ 0.2397{\rm i} (9.36e$-$04)		&	-0.5567 $\pm$ 3.6097{\rm i} (1.34e$-$03)			&			1.06		&		0.71		\\[.2ex]
\hskip -3ex		&		&			&							&	-0.1440 (8.94e$-$05)		&	-2.9445 $\pm$ 4.8214{\rm i} (1.03e$-$03)				&	&		\\[.2ex]
\hskip -3ex		&  		& 			&							 &  -0.6926 $\pm$ 3.2525{\rm i} (3.47e$-$05) 		&  -0.1151 $\pm$ 0.2397{\rm i} (9.36e$-$04)    &      &   \\[.3em]
\hline
\hskip -3ex  	&		&		&									&	-0.3179 $\pm$ 1.0437{\rm i} (7.39e$-$02)		&	-17.0961 $\pm$ 0.0919{\rm i} (8.94e$-$02)	&	&		\\[.2ex]
\hskip -3ex 	&		&			&								&	-0.3464 $\pm$ 0.5796{\rm i} (5.54e$-$02)		&	-0.3179 $\pm$ 1.0437{\rm i} (7.39e$-$02)	&		&		\\[.2ex]
\hskip -3ex 10 	&	 \texttt{mimo8x8\_system}		&	13309,8,8		&	3		&	-0.7168 $\pm$ 0.1238{\rm i} (2.01e$-$02)			&	-0.3464 $\pm$ 0.5796{\rm i} (5.54e$-$02)		&			0.62		&		2.11		\\[.2ex]
\hskip -3ex		&		&			&							&	-9.3974  (1.50e$-$02)		&	-16.7938  (4.48e$-$02)				&	&		\\[.2ex]
\hskip -3ex		&  		& 			&							 &  -11.2079 (1.15e$-$02) 		&  -0.8185 $\pm$ 0.5909{\rm i} (3.05e$-$02)	    &      &   \\[.3em]
\hline
\hskip -3ex  	&		&		&									&	-0.0073  (6.10e$+$02)	&	-0.1170 $\pm$ 0.2746{\rm i} (8.85e$+$02)	&	&		\\[.2ex]
\hskip -3ex 	&		&			&								&	-0.0106  (4.86e$+$02)	&	-0.0073  (6.10e$+$02)		&		&		\\[.2ex]
\hskip -3ex 11 	&	  \texttt{juba40k}		&	40337,2,1		&	7		&	-0.0016  (3.83e$+$02)		&	-0.0106  (4.86e$+$02)			&			2.70		&		2.57		\\[.2ex]
\hskip -3ex		&		&			&							&	 -0.0078  (3.77e$+$02)	&	-0.0016  (3.83e$+$02)				&	&		\\[.2ex]
\hskip -3ex		&  		& 			&							 &      -0.0048  (1.59e$+$01) 		&     -0.0078  (3.77e$+$02)    &      &   \\[.3em]
\end{tabular}
\end{table}

The runtimes of the two algorithms listed in the last column of Table \ref{tab:results_bm} are similar; one or the 
other is a little faster in some cases, but there does not appear any substantial difference in the runtimes. 
A better insight into the runtimes can be obtained from Table \ref{tab:results_bm2}. It is apparent
from this table that Algorithm \ref{alg:SM} performs fewer number of LU factorizations consistently.
On the other hand, SAMDP usually requires fewer number of linear system solves. One main factor
that determines which algorithm has a better runtime is the computational cost of an LU factorization as  
compared to that of solving triangular systems. On examples where LU factorization computations
are considerably expensive, it is reasonable to expect that Algorithm \ref{alg:SM} would have a better runtime. 
But, as the benchmark examples are sparse and even nearly banded at times, it is sometimes
the case that LU factorizations are cheap to obtain, only at a cost of a small constant times that of 
triangular systems. Such systems seem to favor SAMDP. This seems to be that case for instance 
with the \texttt{bips07\_3078} example. The second example for which SAMDP has notably smaller
runtime is the \texttt{mimo8x8\_system} example with $m = p = 8$; as $m$ and $p$ are relatively
larger, Algorithm \ref{alg:SM} needs to solve quite a few triangular systems causing the difference
in the runtimes. In general, the larger $m$ and $p$ are, the more triangular systems need to
be solved and more computational work is required by Algorithm \ref{alg:SM}, even though it should 
still converge quickly in a few iterations.

\begin{table} 
 \caption{ The number of LU factorizations ($\#\,$LU) and number of linear system solves ($\#\:$lin sol) 
 performed by SAMDP and Algorithm \ref{alg:SM} on 11 benchmark examples are listed. Additionally,
 the number of restarts ($\#\,$res) for SAMDP, and the subspace dimension at termination (sdim),
 the time for the construction of the initial subspaces in seconds (init sub) for Algorithm \ref{alg:SM} are provided.  } 
 \label{tab:results_bm2}
\scriptsize
 \begin{tabular}{c|c|ccc|cccc}
		 		& 		 &    \multicolumn{3}{c|}{SAMDP \cite{RM2006b}} &  \multicolumn{4}{c}{Alg.~\ref{alg:SM}} \\[.3em]
		  $\#$ 	& example &	$\#$ LU	&	$\#$ lin sol	&	$\#$ res 				&	$\#$ LU	&	$\#$ lin sol		&	sdim		&	init sub		 \\[.2em]
\hline
1 	&	 \texttt{CDplayer}			&	25	&	40	&			0				&		10			&		80		&	40			&	0.01		\\[.2em]
2 	&	 \texttt{iss} 				&	21	&	34	&			0				&		11			&		132		&	66			&	0.01		\\[.2em]
3 	&	  \texttt{S40PI\_n}			&	41	&	64	&			1				&		15			&		60		&	30			&	0.04		\\[.2em]
4 	&	  \texttt{S80PI\_n}			&	43	&	67	&			1				&		20			&		80		&	40			&	0.09		\\[.2em]
5 	&	  \texttt{M10PI\_n}			&	53	&	82	&			2				&		18			&		216		&	108			&	0.03		\\[.2em]	
6 	&	  \texttt{bips07\_1998}		&	52	&	80	&			2				&		22			&		352		&	176			&	2.17		\\[.2em]
7 	&	 \texttt{bips07\_3078}		&	41	&	64	&			1				&		23			&		368		&	184			&	0.92		\\[.2em]
8 	&	\texttt{xingo\_afonso\_itaipu}	&	35	&	55	&			1				&		26			&		104		&	52			&	0.30		\\[.2em]
9 	&	 \texttt{ww\_vref\_6405}		&	61	&	94	&			2				&		22			&		88		&	44			&	0.29		\\[.2em]
10 	&	 \texttt{mimo8x8\_system}		&	29	&	46	&			0				&		17			&		544		&	272			&	1.04		\\[.2em]
11 	&	  \texttt{juba40k}			&	51	&	79	&			1				&		24			&		144		&	48			&	0.92		\\[.2em]
\end{tabular}
\end{table}

In any case, the main conclusion that can be drawn from these experiments is that Algorithm \ref{alg:SM}
is more robust than SAMDP in converging to the most dominant poles. SAMDP,
together with some of the most dominant poles, seems to converge also some of the less dominant poles.


\subsection{Effect of Initial Subspace Dimension and Number of Dominant Poles Sought on Runtime}\label{sec:ndpsdim}
Here, we illustrate how the initial subspace dimension and desired number of dominant poles affect the runtime
of Algorithm \ref{alg:SM} on the Brazilian power plant example \texttt{bips98\_1450} with $n = 15066$, $m = p = 4$. 
The parameter settings are as in the previous subsection except, in one of the experiments, we vary the initial subspace dimension,
keeping the number of dominant poles fixed at 5, and, in the second experiment, the number of dominant
poles is varied while keeping the number of initial interpolation points fixed at 15.

The results of the first experiment are displayed in the left-hand plot in Figure \ref{fig:nde_sdim}.
The initial subspace dimension is increased by the way of increasing the number of initial interpolation points.
As usual, the initial interpolation points are selected as the dominant poles of a reduced system obtained
from an application of the algorithm in \cite{Aliyev2017}. Initially, when number of initial interpolation points is about 8-12, 
the runtime does not change by much, but then it increases more or less monotonically.  The reason is that the
number of subspace iterations is already small at about 3-4 with the number of initial interpolations
about 8-12. For larger number of interpolation points, the number of subspace iterations
still remains about 3-4, but there is the additional cost of interpolating at further points in the form
of computing further LU factorizations and solving triangular systems.  This is a behavior we typically
observe on the benchmark examples. A small number of initial interpolation points is usually
sufficient for accuracy at a reasonable computational burden.

In the second experiment whose results are shown on the right-hand side in Figure \ref{fig:nde_sdim},
the runtime of Algorithm \ref{alg:SM} usually increases as the number of dominant poles sought is increased.
Only in this example we set the termination tolerance on the residuals as $\texttt{tol} = 10^{-6}$
to avoid convergence difficulties for larger number of dominant poles.
The increase in the runtime is in harmony with the increase in the number of LU factorizations performed,
also depicted in the plot. The plot of the number of triangular systems solved, the other important factor
affecting the runtime, with respect to the number of dominant poles is omitted, as it resembles the
one for LU factorizations scaled up by about a factor of 16. This behavior of runtime as a function
of number of dominant poles is expected in general, but how quickly the runtime and number of LU factorizations
grow vary from example to example. In this particular example, the runtime of SAMDP grows faster than
that of Algorithm \ref{alg:SM} with respect to the number of dominant poles.

Finally, the computed ten most dominant poles of the \texttt{bips98\_1450} example together with
the dominance metrics (inside parentheses) are
\begin{equation*}
	\begin{split}
&	\lambda_1		\;	=	\;	-0.5974 \pm 5.4850{\rm i} \, (1.52{\rm e} + 02), 		\quad\;\;		\lambda_6		\;	=	\;	-1.2726 \pm 9.8846{\rm i} \, (2.51{\rm e}+01),  \\
&	\lambda_2		\;	=	\;	-0.4692 \pm 4.5838{\rm i} \, (1.03{\rm e} + 02), 		\quad\;\;		\lambda_7		\;	=	\;	-30.5404 \pm 9.7736{\rm i} \, (2.23{\rm e}+01),  \\	
&	\lambda_3		\;	=	\;	-7.3723 \pm 2.7152{\rm i} \, (8.34{\rm e} + 01), 		\quad\;\;		\lambda_8		\;	=	\;	-2.8731 \pm 4.0364{\rm i} \, (2.14{\rm e}+01),  \\
&	\lambda_4		\;	=	\;	-0.9818 \pm 6.3822{\rm i} \, (2.96{\rm e} + 01), 		\quad\;\;		\lambda_9		\;	=	\;	-0.6466 \pm 3.6325{\rm i} \, (1.77{\rm e}+01), \\
&	\lambda_5		\;	=	\;	-2.9203 \pm 5.4850{\rm i} \, (2.69{\rm e} + 01), 		\quad\;\;		\lambda_{10}	\;	=	\;	-2.7973 \pm 11.0533{\rm i} \, (1.54{\rm e}+01), 
	\end{split}
\end{equation*}
where $\lambda_j$ denotes the $j$th most dominant pole.

\begin{figure} 
 \centering
	\begin{tabular}{ll}
		\hskip -3ex
			\includegraphics[width = .51\textwidth]{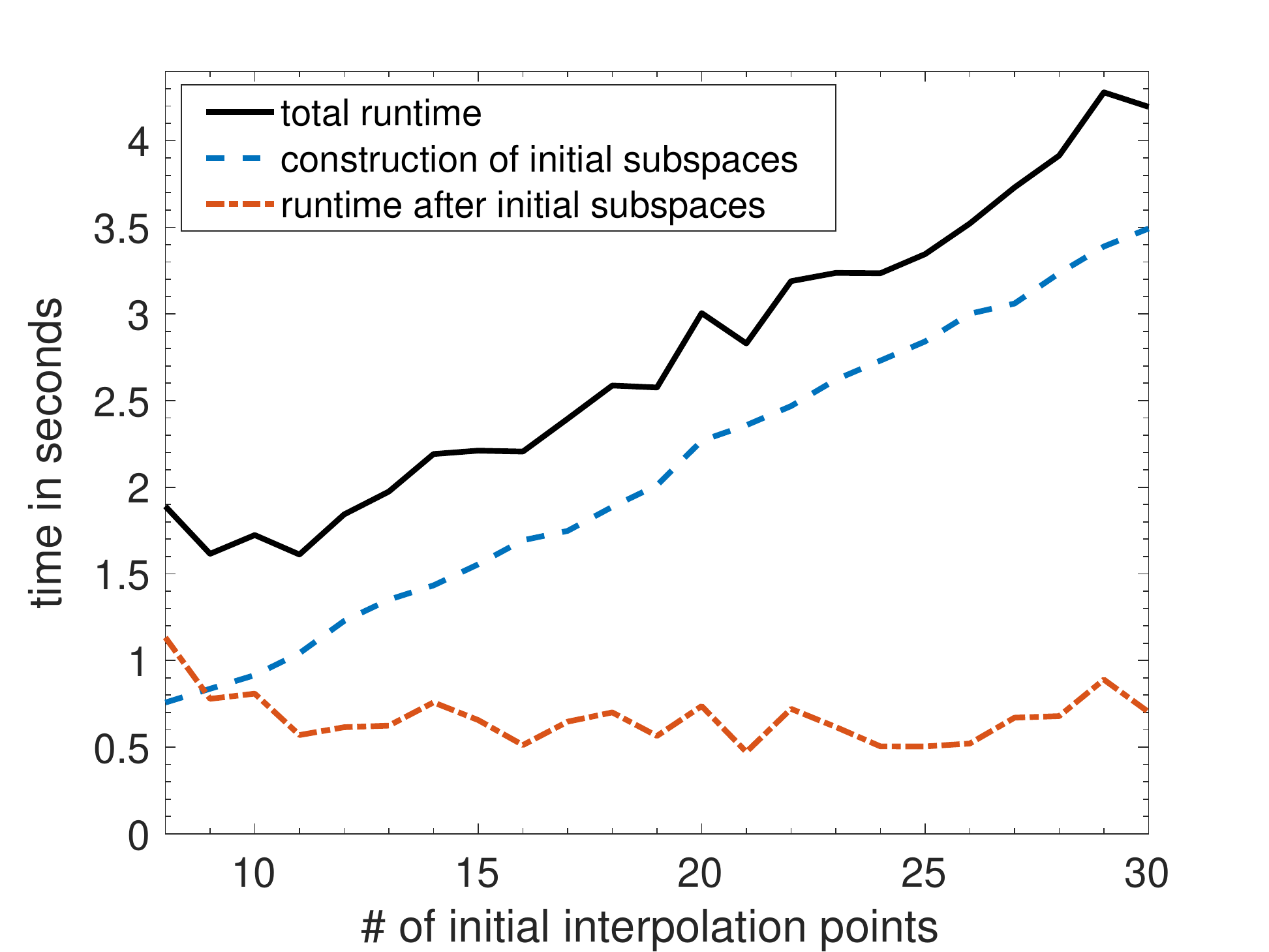} & 
			\includegraphics[width = .51\textwidth]{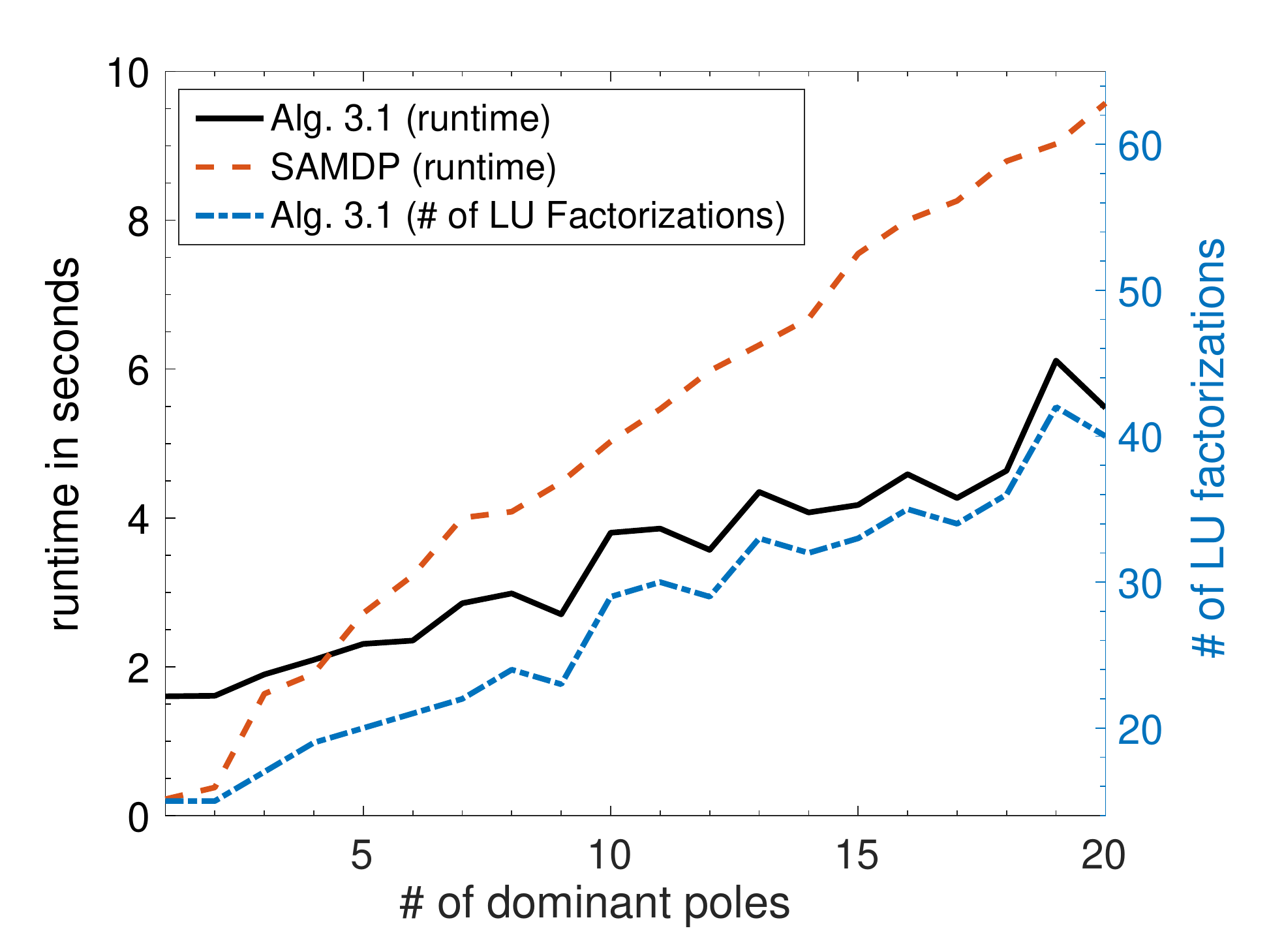}
	\end{tabular}
		\caption{  	The performance on the \texttt{bips98\_1450} example as a function of
		number of initial interpolation points (left), and
		as a function of the number of dominant poles sought (right).}
		\label{fig:nde_sdim}
\end{figure}

\section{Concluding Remarks}
The dominant poles of the transfer function of a descriptor system is used in the literature for model
order reduction. Additionally, dominant poles provide information about how the transfer function
behaves when it is restricted to the imaginary axis, in particular about regions on the imaginary axis 
where the transfer function attains large norm. Hence, it is plausible to initialize the
algorithms for large-scale ${\mathcal L}_\infty$-norm computation based on information
retrieved from dominant poles.

Here, we have proposed an interpolatory subspace framework to estimate a prescribed number
of dominant poles for a descriptor system. At every iteration, the dominant poles of
a projected small problem is computed using standard eigenvalue solvers such as the ones
based on the QZ algorithm. Then the projection subspaces are expanded so that the 
transfer function of the projected problem after expansion Hermite interpolates the 
original transfer function at these computed dominant poles. We have shown that
the proposed framework converges at least at a quadratic rate under mild assumptions,
and verified this result on real benchmark examples. Our numerical experiments indicate
that on benchmark examples the framework locates the dominant poles more reliably in 
comparison to SAMDP \cite{RM2006b}, one of the existing methods for dominant pole estimation.

It may be possible to extend the framework introduced here to more general class of
transfer functions beyond rational functions, such as the transfer functions associated 
with delay systems. 
Moreover, a careful incorporation of the framework here for 
dominant pole estimation to initialize the algorithms for large-scale ${\mathcal L}_\infty$-norm 
computation may be an important step. It may pave the way for accurate and efficient
computation of ${\mathcal L}_\infty$ norm for descriptor systems of large order. 

\smallskip

\textbf{Software.} A Matlab implementation of Algorithm \ref{alg:SM} taking into account the
practical issues discussed in Section \ref{sec:prac_details} is publicly available on the web at
\url{https://zenodo.org/record/5103430}.

This implementation can be run on the benchmark examples in Section \ref{sec:numexp_bench}
using the script \texttt{demo\_on\_benchmarks}. 


\smallskip

\textbf{Acknowledgements.} The author is grateful to two anonymous referees who provided
invaluable feedback on the initial version of this manuscript.


\bibliography{DP}

\end{document}